\documentclass[a4paper]{amsart}
\usepackage{amsmath}
\usepackage{amssymb}

\theoremstyle{plain}
\newtheorem{theorem}{\bf Theorem}[section]
\newtheorem{proposition}[theorem]{\bf Proposition}
\newtheorem{lemma}[theorem]{\bf Lemma}

\theoremstyle{definition}

\newtheorem{definition}[theorem]{\bf Definition}

\newcommand{\N}{\mathbb N}
\newcommand{\Z}{\mathbb Z}

\newcommand{\C}{\mathcal C}
\newcommand{\A}{\mathcal A}
\newcommand{\Dc}{\mathcal D}
\newcommand{\Fc}{\mathcal F}
\newcommand{\vp}{\mathsf v}

\newcommand{\ord}{\text{\rm ord}}

\newcommand{\supp}{\text{\rm supp}}
\newcommand{\Ker}{\text{\rm Ker}}
\newcommand{\Wo}{W_0^{(2)}}

\renewcommand{\t}{\, | \,}

\newcommand{\be}{\begin{equation}}
\newcommand{\ee}{\end{equation}}

\newcommand{\bnml}{\begin{multline}}
\newcommand{\enml}{\end{multline}}
\newcommand{\buml}{\begin{multline*}}
\newcommand{\euml}{\end{multline*}}

\newcommand{\ber}{\begin{eqnarray}}
\newcommand{\eer}{\end{eqnarray}}

\newcommand{\sig}{\widetilde{\sigma}}

\begin{document}

\title{Inverse zero-sum problems {III}}

\address{Center for Combinatorics, Nankai University, Tianjin
300071, P.R. China} \email{wdgao\_1963@yahoo.com.cn}

\address{Institut f\"ur Mathematik und Wissenschaftliches Rechnen \\
Karl-Franzens-Universit\"at Graz \\
Heinrichstra\ss e 36\\
8010 Graz, Austria} \email{alfred.geroldinger@uni-graz.at,
diambri@hotmail.com}

\author{Weidong Gao, Alfred Geroldinger, and David J. Grynkiewicz}




\maketitle

\bigskip
\section{Introduction} \label{1}
\bigskip

We continue the investigations started in \cite{Ga-Ge-Sc07a, Sc08e}.
Let \ $G = C_n \oplus C_n$ \ with $n \ge 2$. We say that $G$ has
Property {\bf B} if every minimal zero-sum sequence $S$ over $G$ of
length $|S| = 2n-1$ contains an element with multiplicity $n-1$. The
aim of the present paper is to prove the following two results.

\medskip
{\bf Theorem.} {\it Let \ $G = C_{mn} \oplus C_{mn}$ \ with \ $m,\,
n \ge 3$ odd and $mn>9$. If both \ $C_m \oplus C_m$ \ and \ $C_n \oplus C_n$
\ have Property {\bf B}, then \ $G$ \ has Property {\bf B}.}

\medskip
{\bf Corollary.} {\it Let \ $G = C_{n_1} \oplus C_{n_2}$ \ with $1 <
n_1 \t n_2$, and suppose that, for every prime divisor $p$ of $n_1$,
the group \ $C_p \oplus C_p$ has Property {\bf B}. Then \ $C_{n_1}
\oplus C_{n_1}$ has Property {\bf B}, and a sequence $S$ over $G$ of
length $\mathsf D (G) = n_1+n_2-1$ is a minimal zero-sum sequence if
and only if it has one of the following two forms{\rm \,:}
\begin{itemize}
\medskip
\item \[
      S = e_j^{\ord (e_j)-1} \prod_{\nu=1}^{\ord (e_k)}
      (x_{\nu}e_j+e_k) \,, \qquad \text{where}
      \]
      $(e_1, e_2)$ is a basis of $G$ with $\ord (e_i) = n_i$ for $i \in \{1,2\}$,
      $\{j,k\} = \{1,2\}$, $x_1, \ldots, x_{\ord (e_k)}  \in
      [0, \ord (e_j)-1]$, and $x_1 + \ldots + x_{\ord (e_k)} \equiv 1
      \mod \ord (e_j)$.

\medskip
\item \[
      S = g_1^{sn_1 - 1} \prod_{\nu=1}^{n_2 + (1-s)n_1} ( -x_{\nu} g_1 +
      g_2) \,, \qquad \text{where}
      \]
      $\{g_1, g_2\}$ is a generating set of $G$ with $\ord (g_2) =
      n_2$, $x_1, \ldots, x_{n_2 + (1-s)n_1} \in [0, n_1-1]$, $x_1 + \ldots + x_{n_2 +
      (1-s)n_1} = n_1-1$, $s \in [1, n_2/n_1]$, and either $s=1$ or
      $n_1g_1 = n_2g_2$.
\end{itemize}
}

\medskip
Thus Property {\bf B} is multiplicative, and if \ $G = C_{n_1}
\oplus C_{n_2}$ \ with $1 < n_1 \t n_2$ is a group of rank two, and
for every prime divisor $p$ of $n_1$ the group \ $C_p \oplus C_p$
has Property {\bf B}, then the minimal zero-sum sequences of maximal
length over $G$ are explicitly characterized.

In Section \ref{2}, we fix our notation and gather the necessary
tools (apart from former work on Property {\bf B} and classical
addition theorems, we use a confirmed conjecture of  Y. ould
Hamidoune, see Theorem \ref{ham-result}). Section \ref{3} contains
some straightforward lemmas. The proof of the Theorem consists of
two major parts: the first is given in Section \ref{4} and the
second, more involved one, is given in Section \ref{5}.

The  Corollary is mainly based on the Theorem above, on former work
of the authors \cite{Ga-Ge03b}, and on recent work by Wolfgang A.
Schmid \cite{Sc08e}. Its proof only needs a few lines and is given
in Section \ref{6}.

\bigskip
\section{Preliminaries} \label{2}
\bigskip

Our notation and terminology are consistent with \cite{Ga-Ge-Sc07a}
and \cite{Ge-HK06a}. We briefly gather some key notions and fix the
notation concerning sequences over  abelian groups. Let $\mathbb N$
denote the set of positive integers and let $\mathbb N_0 = \mathbb N
\cup \{ 0 \}$. For real numbers $a, b \in \mathbb R$, we set $[a, b]
= \{ x \in \mathbb Z \mid a \le x \le b\}$.
Throughout, all abelian groups will be written additively. For $n
\in \mathbb N$, let $C_n$ denote a cyclic group with $n$ elements.
Let \ $G$ \ be an abelian group.

Let \ $A,\, B \subset G$ \ be nonempty subsets. Then \ $A+B = \{a+b
\mid a \in A, b \in B \}$ \ denotes their \ {\it sumset} \ and \
$A-B = \{a-b \mid a \in A, b \in B \}$ \ their \ {\it difference
set}. The \ {\it stabilizer} \ of $A$ is defined as \ $\text{\rm
Stab} (A) = \{ g \in G \mid g +A = A\}$, and $A$ is called \ {\it
periodic} \ if $\text{\rm Stab} (A) \ne \{0\}$.

An $s$-tuple $(e_1, \ldots, e_s)$ of elements of $G$ is said to be \
{\it independent} \ if $e_i \ne 0$ for all $i \in [1,s]$ and, for
every $s$-tuple $(m_1, \ldots, m_s) \in \mathbb Z^{s}$,
\[
m_1 e_1 + \ldots + m_s e_s  =0 \qquad \text{implies} \qquad m_1e_1 =
\ldots = m_se_s = 0 \,.
\]
An $s$-tuple $(e_1, \ldots, e_s)$ of elements of $G$ is called a \
{\it basis} \ if it is independent and $G = \langle e_1\rangle
\oplus \ldots \oplus \langle e_s \rangle$.

Let \ $G = C_n \oplus C_n$ \ with $n \ge 2$, and let \ $(e_1, e_2)$
\ be a basis of $G$. An endomorphism $\varphi \colon G \to G$ with
\[
\left( \varphi (e_1), \varphi (e_2) \right) = \left( e_1, e_2
\right) \cdot
            \left(
            \begin{matrix}
            a & b \\
            c & d
            \end{matrix}
            \right) \,,
           \qquad \text{where} \qquad a,b,c,d \in \Z \,,
\]
is an automorphism if and only if $(  \varphi (e_1), \varphi (e_2)
)$ is a basis, which is equivalent to $\gcd ( ad-bc, n ) = 1$. If
$f_1 \in G$ with $\ord (f_1) = n$, then clearly  there is an $f_2
\in G$ such that $(f_1, f_2)$ is a basis of $G$.

\smallskip
Let $\mathcal F(G)$ be the free  monoid with basis $G$. The elements
of $\mathcal F(G)$ are called \ {\it sequences} \ over $G$. We write
sequences $S \in \mathcal F (G)$ in the form
\[
S =  \prod_{g \in G} g^{\mathsf v_g (S)}\,, \quad \text{with} \quad
\mathsf v_g (S) \in \mathbb N_0 \quad \text{for all} \quad g \in G
\,.
\]
We call \ $\mathsf v_g (S)$  the \ {\it multiplicity} \ of $g$ in
$S$, and we say that $S$ \ {\it contains} \ $g$ \ if \ $\mathsf v_g
(S) > 0$.  A sequence $S_1 $ is called a \ {\it subsequence} \ of
$S$ \ if \ $S_1 \, | \, S$ \ in $\mathcal F (G)$ \ (equivalently, \
$\mathsf v_g (S_1) \le \mathsf v_g (S)$ \ for all $g \in G$). Given two sequences $S,\,T\in
\Fc(G)$, we denote by $\gcd(S,T)$ the longest subsequence dividing
both $S$ and $T$. If a sequence $S \in \mathcal F(G)$ is written in
the form $S = g_1 \cdot \ldots \cdot g_l$, we tacitly assume that $l
\in \mathbb N_0$ and $g_1, \ldots, g_l \in G$.

\smallskip

For a sequence
\[
S \ = \ g_1 \cdot \ldots \cdot g_l \ = \  \prod_{g \in G} g^{\mathsf
v_g (S)} \ \in \mathcal F(G) \,,
\]
we call
\[
|S| = l = \sum_{g \in G} \mathsf v_g (S) \in \mathbb N_0 \qquad
\text{the \ {\it length} \ of \ $S$}\,,
\]
\[
\begin{split} \mathsf h (S) =  &   \max \{ \mathsf v_g (S) \mid g \in G \} \in [0, |S|]\\
  & \qquad  \text{the \ {\it maximum of the multiplicities} \ of \
$S$}\,,\end{split}
\]
\[
\supp (S) = \{g \in G \mid \mathsf v_g (S) > 0 \} \subset G \qquad
\text{the \ {\it support} \ of \ $S$}\,,
\]
\[
\sigma (S) = \sum_{i = 1}^l g_i = \sum_{g \in G} \mathsf v_g (S) g
\in G \qquad \text{the \ {\it sum} \ of \ $S$}\,,
\]
\[
\begin{split}
\Sigma_k (S)  = &  \Bigl\{ \sum_{i \in I} g_i  \Bigm|  I \subset [1,
l]
\ \text{with} \ |I| = k \,\Bigr\} \\
 &\qquad  \text{ the \ {\it set of
$k$-term subsums} \ of \ $S$}\,,  \ \text{for all} \ k \in \mathbb
N\,,\end{split}
\]
\[
\Sigma_{\le k} (S) = \bigcup_{j \in [1,k]} \Sigma_{j} (S) \,, \qquad
\Sigma_{\ge k} (S) = \bigcup_{j \ge k} \Sigma_{j} (S) \,,
\]
and
\[
\Sigma (S)  = \Sigma_{\ge 1} (S)  \ \text{ the \ {\it set of
$($all$)$ subsums} \ of \ $S$} \,.
\]
The sequence \ $S$ \ is called
\begin{itemize}
\item {\it zero-sum free} \ if \ $0 \notin \Sigma (S)$,

\item a {\it zero-sum sequence} \ if \ $\sigma (S) = 0$,

\item a {\it minimal zero-sum sequence} \ if $1 \ne S$, $\sigma(S)=0$,  and every $S'|S$ with $1\leq |S'|<|S|$  is
      zero-sum free.
\end{itemize}

\smallskip
We denote by \ $\mathcal A (G) \subset \mathcal F (G)$ \ the set of
all minimal zero-sum sequences over $G$. Every map of abelian groups
\ $\varphi \colon G \to H$ \ extends to a homomorphism \ $\varphi
\colon \mathcal F (G) \to \mathcal F (H)$ \ where \ $\varphi (S) =
\varphi (g_1) \cdot \ldots \cdot \varphi (g_l)$. We say that
$\varphi$ is \ {\it constant} \  on $S$ if $\varphi(g_1) = \ldots =
\varphi(g_l)$. If $\varphi$ is a homomorphism, then $\varphi (S)$ is
a zero-sum sequence if and only if $\sigma (S) \in \Ker (\varphi)$.

\medskip
\begin{definition} \label{2.1}
Let \ $G$ \ be a finite abelian group with exponent $n$.
\begin{enumerate}
\item Let \ $\mathsf D (G)$ \ denote the smallest integer $l \in \mathbb
      N$ such that every sequence \ $S \in \mathcal F (G)$ \ of length
      $|S| \ge l$ has a zero-sum subsequence. Equivalently, we have \ $\mathsf D (G) = \max \{ |S| \mid S \in
      \mathcal A (G) \}$), and  $\mathsf D
      (G)$ is called the \ {\it Davenport constant} \ of $G$.

\item Let \ $\eta (G)$ \ denote the smallest integer $l \in \mathbb
      N$ such that every sequence \ $S \in \mathcal F (G)$ \ of length
      $|S| \ge l$ has a zero-sum subsequence $T$ of length $|T| \in
      [1,n]$.

\smallskip
\item We say that \ $G$ \ has Property {\bf C} if every sequence \ $S$
      \ over $G$ of length
      $|S| = \eta (G) - 1$, with no
      zero-sum subsequence of length in $[1, n]$, has the form \ $S = T^{n-1}$ \ for some sequence \ $T$ \ over $G$.
\end{enumerate}
\end{definition}

\medskip
\begin{lemma} \label{2.2}
Let \ $G = C_{n_1} \oplus C_{n_2}$ \ with \ $1 \le n_1 \t
      n_2$.
\begin{enumerate}
\item We have \ $\mathsf D (G) = n_1+n_2-1$ \ and \ $\eta (G) = 2n_1 + n_2 - 2$.

\smallskip
\item If \ $n_1 = n_2$ \ and \ $G$ \ has Property {\bf B}, then \ $G$ \ has Property {\bf
      C}.
\end{enumerate}
\end{lemma}

\begin{proof}
1. See \cite[Theorem 5.8.3]{Ge-HK06a}.

2. See \cite[Theorem 6.2]{Ga-Ge03b} and \cite[Theorem
6.7.2.(b)]{Ga-Ge06b}.
\end{proof}

Results on $\eta (G)$ for groups of higher rank may be found in
recent work of C. Elsholtz et.al. (\cite{El04,E-E-G-K-R07}).

\medskip
\begin{lemma} \label{2.3}
Let $G = C_n \oplus C_n$ with \ $n \ge 2$.
\begin{enumerate}
\item Then the following statements are equivalent\,{\rm :}
      \begin{enumerate}
      \item If \ $S \in \mathcal F (G)$, \ $|S| = 3n-3$ and $S$ has
            no zero-sum subsequence $T$ of length $|T| \ge n$, then there exists
            some $a \in G$ such that \ $0^{n-1} a^{n-2} \t S$.

      \smallskip
      \item If \ $S \in \mathcal F(G)$ \ is zero-sum free and \ $|S| = 2n-2$,
            then \ $a^{n-2} \t S$ for some $a \in G$.

      \smallskip
      \item If \ $S \in \mathcal A(G)$ and \ $|S| = 2n-1$, then \ $a^{n-1} \t S$
            for some $a \in G$.

      \smallskip
      \item If \ $S \in \mathcal A(G)$ and \ $|S| = 2n-1$, then
            there exists a basis \ $(e_1, e_2)$ of \ $G$ and integers \ $x_1,
            \ldots , x_n \in [0, n-1]$, with $x_1 + \ldots + x_n \equiv 1 \mod
            n$, such that
            \[
            S = e_1^{n-1} \prod_{\nu=1}^n (x_{\nu} e_1 + e_2)\,.
            \]
      \end{enumerate}

\smallskip
\item Let $S \in \mathcal A (G)$ be of length $|S| = 2n-1$
      and $e_1 \in G$ with $\mathsf v_{e_1} (S)  = n-1$.
      If $(e_1, e_2')$ is a basis of $G$, then there exist some $b \in [0,
      n-1]$ and $a_1', \dots , a_n' \in [0, n-1]$, with $\gcd (b,n) = 1$ and
            $\sum_{\nu=1}^n a_{\nu}' \equiv 1 \mod n$,  such that
            \[
            S = e_1^{n-1}  \prod_{\nu=1}^n (a_{\nu}'e_1 + b e_2') \,.
            \]

\smallskip
\item If \ $S \in \mathcal A (G)$ \ has length $|S| = 2n-1$, then \
      $\ord (g) = n$ \ for all \ $g \in \supp (S)$.
\end{enumerate}
\end{lemma}

\begin{proof}
1. See \cite[Theorem 5.8.7]{Ge-HK06a}.

\smallskip
2. This follows easily from 1; for details see \cite[Proposition
4.1]{Ga-Ge03b}.

\smallskip
3. See \cite[Theorem 5.8.4]{Ge-HK06a}.
\end{proof}

\smallskip
The characterization in Lemma \ref{2.3}.1 gives rise to the following
definition.

\medskip
\begin{definition} \label{2.4}
Let \ $G = C_n \oplus C_n$ \ with $n \ge 2$.
\begin{enumerate}
\item Let $\Upsilon(G)$ be the set of all $S\in \mathcal A (G)$ for which there exists a basis
      \ $(e_1, e_2)$ of \ $G$ and integers \ $x_1,
      \ldots , x_n \in [0, n-1]$, with $x_1 + \ldots + x_n \equiv 1 \mod
      n$, such that \ $S = e_1^{n-1} \prod_{\nu=1}^n (x_{\nu} e_1 + e_2)$.

\smallskip
\item Let $\Upsilon_u(G)$ be the set of those $S\in \Upsilon(G)$
      with a unique term of multiplicity $n-1$, and let
      $\Upsilon_{nu}(G)=\Upsilon(G)\setminus \Upsilon_u(G)$.
\end{enumerate}
\end{definition}

\smallskip
Thus, by Lemma \ref{2.3}.1, a group \ $G = C_n \oplus C_n$ \ with $n
\ge 2$ has Property {\bf B} if and only if \ $\mathcal A (G) =
\Upsilon(G)$.

\medskip
\begin{lemma} \label{2.5}
Let \ $G = C_{mn} \oplus C_{mn}$ \ with \ $m, n \ge 2$, let \ $S\in
{\mathcal A} (G)$ \ be of length \ $|S|=2mn-1$, and  let $\varphi
\colon G \to G$ denote the multiplication by $m$ homomorphism.
\begin{enumerate}
\item $\varphi (S)$ is not a product of $2m$ zero-sum subsequences.
      Every  zero-sum   subsequence \ $T$ \ of $\varphi (S)$ of length \ $|T| \in [1,n]$ \ has length $n$, and
      $0 \notin \supp ( \varphi (S) )$.

\smallskip
\item $S$ may be written in the form $S = W_0 \cdot \ldots \cdot W_{2m-2}$,
      where $W_0, \ldots, W_{2m-2} \in \mathcal F (G)$ with $|W_{0}| = 2n-1$,
      $|W_1| = \ldots = |W_{2m-2}| = n$ and
      $\sigma (W_0), \ldots, \sigma (W_{2m-2}) \in \Ker (    \varphi )$.
\end{enumerate}
\end{lemma}

\begin{proof}
See \cite[Lemma 3.14]{Ga-Ge03b}.
\end{proof}

\medskip
The following is the Erd\H{o}s-Ginzburg-Ziv Theorem  and the
corresponding characterization of extremal sequences.

\medskip
\begin{theorem}\label{EGZ-thm}
Let $G$ be a cyclic group of order $n \ge 2$ and $S\in \Fc(G)$.
\begin{enumerate}
\item If $|S|\geq 2n-1$, then $0\in \Sigma_n(S)$.

\smallskip
\item If $|S|=2n-2$ and $0\notin \Sigma_n(S)$, then $S=g^{n-1}h^{n-1}$ for some $g,\,h\in G$ with $\ord (g-h) = n$.
\end{enumerate}
\end{theorem}

\begin{proof}
1. See \cite[Corollary 5.7.5]{Ge-HK06a} or \cite[Theorem
2.5]{Na96b}.

\smallskip
2. See \cite[Lemma 4]{Bi-Di92} for one of the original proofs, and
\cite[Section 7.A]{Ge08c}.
\end{proof}

\smallskip
The following result was  a conjecture of Y. ould Hamidoune
\cite{Ha03b} confirmed in \cite[Theorem 1]{Gr05b}.

\medskip
\begin{theorem}\label{ham-result}
Let \ $G$ \ be a finite abelian group, $S\in \Fc(G)$ of length \
$|S|\ge |G|+1$, and $k \in \N$ with  $k\leq |\supp(S)|$. If \
$\mathsf h(S)\leq |G|-k+2$ \ and \ $0\notin \Sigma_{|G|}(S)$, then \
$|\Sigma_{|G|}(S)| \geq |S|-|G|+k-1$.
\end{theorem}

\bigskip
\section{Preparatory Results.} \label{3}
\bigskip

We first prove several lemmas determining in what ways a sequence
$S\in \Upsilon(C_m \oplus C_m)$, where $m \ge 4$, can be slightly
perturbed and still remain in $\Upsilon(C_m \oplus C_m)$. These will
later be heavily used in Section \ref{5}, always in the setting where
$K=\Ker(\varphi)$ and $\varphi \colon G \to G$ is the
multiplication by $m$ map.

\medskip
\begin{lemma} \label{lem-Pertebation-I}
Let \ $K=C_m\oplus C_m$ \ with $m \ge 4$, let \ $g \in K$, \ and let
\ $S=f_1^{m-1}\prod_{\nu=1}^{m}(x_{\nu} f_1+f_2) \in \Upsilon_u(K)$
\ with $x_1, \ldots, x_m \in \Z$.
\begin{enumerate}
\item If $S'=f_1^{-2}S(f_1+g)(f_1-g)\in \Upsilon(K)$, then  $g=0$ and hence $S=S'$.

\smallskip
\item If $S'=f_1^{-1}(x_jf_1+f_2)^{-1}S(f_1+g)(x_jf_1+f_2-g)\in \Upsilon(K)$, then  $g\in \{0,\,(x_j-1)f_1+f_2\}$
      and hence $S=S'$.

\smallskip
\item If $S'=(x_jf_1+f_2)^{-1}(x_kf_1+f_2)^{-1}S(x_jf_1+f_2+g)(x_kf_1+f_2-g)\in \Upsilon(K)$
      with $j, k \in [1,m]$ distinct, then $g\in \langle f_1 \rangle$.
\end{enumerate}
\end{lemma}

\begin{proof}
1. Assume to the contrary that $g \ne 0$ and thus $S \ne S'$. Then
$\mathsf v_{f_1} (S') < m-1$ and, since $S \in \Upsilon_u (K)$, it
follows that there is some $j \in [1,m]$ such that
$(x_jf_1+f_2)^{m-1} \t S'$, $(x_jf_1+f_2)^{m-3} \t S$, and
$x_jf_1+f_2 = f_1+g$. If we set $f_2' = x_jf_1+f_2$, then $S =
f_1^{m-1} \prod_{\nu=1}^m \bigl( (x_{\nu} - x_j)f_1 + f_2' \bigr)$,
and thus we may assume that $f_2 = f_2'$. Then $f_2 = f_1+g$ and
$f_1-g = f_2-2g = 2f_1-f_2$. Since $m \ge 4$, it follows that $f_1
\t S'$. Since $S'\in \Upsilon(K)$, $f_2^{m-1}\t S'$ and $f_1,
2f_1-f_2 \in \supp (S') \setminus \{f_2\}$, it follows that
$(2f_1-f_2) - f_1 = f_1-f_2 \in \langle f_2 \rangle$, a
contradiction.

\smallskip
2. After renumbering, we may suppose  that $j=n$. If $f_1^{m-1}|S'$
then $f_1+g=f_1$ or $x_nf_1+f_2-g=f_1$, and $S'=S$. Otherwise,
$f_1^{m-1} \nmid S'$ and we shall derive a contradiction. Observe that we cannot have
$f_1+g=x_nf_1+f_2-g=x_jf_1+f_2$. Thus, since
$S'\in \Upsilon(K)$ and $S\in \Upsilon_u(K)$, it follows that (after
renumbering again if necessary) either
\[
S'=f_1^{m-2}(xf_1+f_2)^{m-1}(x_nf_1+f_2-g)(x_{n-1}f_1+f_2) \quad
\text{ with} \quad f_1+g=xf_1+f_2 \,,
\]
 or
\[
S'=f_1^{m-2}(xf_1+f_2)^{m-1}(f_1+g)(x_{n-1}f_1+f_2) \quad \text{
with}  \quad x_nf_1+f_2-g=xf_1+f_2 \,.
\]
In the first case, we have $(x_nf_1+f_2-g)=(x_n-x+1)f_1$ and hence
$f_1^{m-2}((x_n-x+1)f_1)|S'$. However, since $(x_n-x+1)f_1 =
(x_nf_1+f_2-g) \neq f_1$, it follows that $f_1^{m-2}((x_n-x+1)f_1)$
is not zero-sum free, a contradiction. In the second case,  one can
derive a contradiction similarly.

\smallskip
3. Since $m\geq 3$, $f_1^{m-1} \t S'$ and $S'\in \Upsilon (K)$, it
follows that $(x_jf_1+f_2+g)-(x_lf_1+f_2)\in \langle f_1 \rangle$,
where $l\neq j,\,k$, and hence $g\in \langle f_1 \rangle$.
\end{proof}

\medskip
\begin{lemma} \label{lem-Pertebation-II}
Let \ $K=C_m\oplus C_m$ \ with $m \ge 4$, \ $g \in K$ \ and \ $S =
f_1^{m-1}f_2^{m-1}(f_1+f_2) \in \Upsilon_{nu}(K)$.
\begin{enumerate}
  \item If \ $S'=f_1^{-2}S(f_1+g)(f_1-g)\in \Upsilon(K)$, \ then \ $g\in \langle f_2 \rangle$.

  \item If \ $S'=f_2^{-2}S(f_2+g)(f_2-g)\in \Upsilon(K)$, \ then \ $g\in \langle f_1 \rangle$.

  \item If \ $S'=f_1^{-1}f_2^{-1}S(f_1+g)(f_2-g)\in \Upsilon(K)$, \ then $S=S'$ \ and \ $g\in\{0,\,-f_1+f_2\}$.

  \item If \ $S'=f_1^{-1}(f_1+f_2)^{-1}S(f_1+g)(f_1+f_2-g)\in \Upsilon(K)$, \ then \ $g\in \langle f_2 \rangle$.

  \item If \ $S'=f_2^{-1}(f_1+f_2)^{-1}S(f_2+g)(f_1+f_2-g)\in \Upsilon(K)$, \ then \ $g\in \langle f_1 \rangle$.
\end{enumerate}
\end{lemma}

\begin{proof}
1. Since $f_2^{m-1} \t S'$ and $S'\in \Upsilon(K)$, it follows that
$f_1+g-(f_1+f_2)\in \langle f_2 \rangle$, whence $g \in \langle f_2
\rangle$.

\smallskip
2. Analogous to 1.

\smallskip
3. If $f_1^{m-1} \t S'$ or $f_2^{m-1} \t S'$, the result follows.
Otherwise, $m\ge  4$ and $\mathsf h(S')=m-1$  imply that $m=4$ and
$f_1+g=f_2-g=f_1+f_2$, a contradiction.

\smallskip
4. Since $m \geq 3$, it follows that $f_1 \t S'$. Now we have
$f_2^{m-1} \t S'$ and $S'\in \Upsilon (K)$ so that
$(f_1+f_2-g)-f_1\in \langle f_2 \rangle$, implying $g\in \langle f_2
\rangle$, as desired.

\smallskip
5. Analogous to 4.
\end{proof}

\medskip
\begin{lemma} \label{lem-Pertebation-III}
Let \ $K=C_m\oplus C_m$ \ with $m \ge 4$, \ $g \in K$ \ and \ $S =
f_1^{m-1}f_2^{m-1}(f_1+f_2) \in \Upsilon_{nu}(K)$.
\begin{enumerate}
  \item If \ $S'=f_1^{-2}S(f_1+g)(f_1-g)\in \Upsilon_{nu}(K)$, then
        \ $g=0$, \ and hence \ $S=S'$.

  \item If \ $S'=f_2^{-2}S(f_2+g)(f_2-g)\in \Upsilon_{nu}(K)$, \ then  \ $g=0$, \ and hence \ $S=S'$.

  \item If \ $S'=f_1^{-1}f_2^{-1}S(f_1+g)(f_2-g)\in \Upsilon_{nu}(K)$, \ then \ $g\in\{0,\,-f_1+f_2\}$, and hence $S=S'$.

  \item If \ $S'=f_1^{-1}(f_1+f_2)^{-1}S(f_1+g)(f_1+f_2-g)\in \Upsilon_{nu}(K)$, \ then \ $g\in
        \{0,\,f_2\}$, and hence $S=S'$.

  \item If \ $S'=f_2^{-1}(f_1+f_2)^{-1}S(f_2+g)(f_1+f_2-g)\in \Upsilon_{nu}(K)$, \ then \ $g\in
        \{0,\,f_1\}$, and hence $S=S'$.
\end{enumerate}
\end{lemma}

\begin{proof}
1. Assume to the contrary that $g \ne 0$ and $S \ne S'$. Since $S'
\in \Upsilon_{nu}(K)$ and $m \ge 4$, we get $f_1+g=f_1-g=f_1+f_2$
and hence $-2f_2=2g=0$, a contradiction.

\smallskip
2. - 5. Similar.
\end{proof}
%

Next we prove two simple structural lemmas which will be our all-purpose  tools for turning locally obtained information into global structural conditions on $S$. They are also the reason for the hypothesis of $m$ and $n$ odd in the Theorem.

\medskip
\begin{lemma} \label{lem-basic-exchange}
Let \ $G$ \ be an abelian group,  $a\in G$ with $\ord(a)>2$, and
$S,\,T\in \Fc(G) \setminus \{1\}$ with $|\supp(S)|\geq |\supp(T)|$.
\begin{enumerate}
\item If \ $\supp(S)-\supp(T)=\{0\}$, \ then $S=g^{|S|}$ and $T=g^{|T|}$, for some $g\in G$.

\smallskip
\item If \ $\supp(S)-\supp(T)\subset \{0,a\}$, then \ $S=g^{s}(g+a)^{|S|-s}$ and $T=g^{|T|}$,
 for some $g\in G$ and  $s \in [0, |S|]$.

\smallskip
\item If \ $|S|,\, |T| \geq 2$ and $\bigcup_{i=1}^{2}(\Sigma_i(S)-\Sigma_i(T))\subset
      \{0,a\}$, \
        then either $S=g^{|S|-1}(g+a)$ and $T=g^{|T|}$, or else $S=g^{|S|}$ and $T=g^{|T|}$, for some $g\in G$.
\end{enumerate}
\end{lemma}

\begin{proof}
Note that $\Sigma_1(S)=\supp(S)$ and that all hypotheses imply
$\supp(S)-\supp(T)\subset \{0,a\}$. Since $\ord(a)>2$, it follows
that $\{0,a\}$ contains no periodic subset, and thus Kneser's
Theorem (see e.g., \cite[Theorem 5.2.6]{Ge-HK06a}) implies that
\[
2\geq |\supp(S)-\supp(T)|\geq |\supp(S)|+|\supp(T)|-1 \,.
\]
Therefore we get $|\supp(S)|\leq 2$ and $|\supp(T)|=1$. Items 1 and
2 now easily follow. For the proof of part 3, we apply 2, and thus
we may assume that $\supp(S)\subset \{g,\,(g+a)\}$ and $T=g^{|T|}$.
Now if item 3 is false, then $(g+a)^2 \t S$, whence
$$2a=((g+a)+(g+a))-(g+g)\in\bigcup_{i=1}^{2}(\Sigma_i(S)-\Sigma_i(T))\subset
\{0,a\},$$ contradicting that $\ord(a)>2$.
\end{proof}

\medskip
\begin{lemma} \label{lem-basic-exchange-ii}
Let \ $G$ \ be an abelian group and let $S\in \Fc(G)$.
\begin{enumerate}
\item If $k \in [1,  |S|-1]$ and $|\Sigma_k(S)|\leq 2$, then $|\supp(S)|\leq 2$.

\smallskip
\item If $k \in [2, |S|-2]$ and $|\Sigma_k(S)|\leq 2$ and $\Sigma_k(S)$ is not a coset of a cardinality two subgroup,
      then either $S=g^{|S|}$ or $S=g^{|S|-1}h$, for some $g,\,h\in G$.

\smallskip
\item If $k \in [1, |S|-1]$ and $|\Sigma_k(S)|\leq 1$, then $S=g^{|S|}$  for some $g\in G$.
\end{enumerate}
\end{lemma}

\begin{proof}
1. Assume to the contrary that $|\supp(S)|\geq 3$, and pick three
distinct elements $x, y, z \in \supp (S)$. If $k=|S|-1$, then
$\Sigma_{|S|-1}(S)=\sigma(S)-\Sigma_1(S)$ and hence
$|\Sigma_{|S|-1}(S)|=|\supp(S)|\geq 3$, a contradiction. Therefore $k \leq |S|-2$.
Let $T$ be a subsequence of $(xyz)^{-1}S$
of length $|T|=k-1\leq |S|-3$. Then $\{x,\,y,\,z\}+ \sigma (T)$ is a
cardinality three subset of $\Sigma_k (S)$, a contradiction.

\smallskip
2. By  1, we have $S=g^{s_1}h^{s_2}$, with $s_1, s_2 \in \mathbb
N_0$, $s_1 \ge s_2$  and $g,\,h\in G$ distinct. Assume to the
contrary that  $s_2\geq 2$. Since
$\Sigma_{|S|-k}(S)=\sigma(S)-\Sigma_k(S)$, it suffices to consider
the case $k\leq \frac{1}{2}|S|$, and thus we have $s_1 \geq
\frac{1}{2}|S| \geq k \ge 2$. Hence  the elements $kg$, $(k-1)g+h$
and $(k-2)g+2h$ are all contained in $\Sigma_k(S)$. Thus, since
$|\Sigma_k(S)|\leq 2$ and $g\neq h$, it follows $\ord(h-g)=2$ and
$\Sigma_k(S)=kg+\{0,h-g\}$, contradicting that $\Sigma_k(S)$ is not
a coset of a cardinality two subgroup.

\smallskip
3. If the conclusion is false, there are distinct $x,y\in G$ with
$xy|S$, and then $\{x,y\}+\sigma(S')$ is a cardinality two subset of
$\Sigma_k(S)$ for any $S'|(xy)^{-1}S$ with $0\leq |S'|=k-1\leq
|S|-2$.
\end{proof}

\bigskip
\section{On the Structure of $\varphi (S)$} \label{4}
\bigskip

\begin{definition} \label{4.1}
Let \ $G = C_{mn} \oplus C_{mn}$ \ with \ $m,\, n \ge 2$, let \ $S \in
\mathcal A (G)$ with $|S| = 2mn-1$, and let  $\varphi \colon G \to
G$ be the multiplication by $m$ homomorphism. Let
\[
\begin{aligned}
\Omega' (S) = \Omega' = \{ (W_0, \ldots, W_{2m-2}) \in \mathcal F
(G)^{2m-1} \mid & \ S = W_0 \cdot \ldots \cdot W_{2m-2}, \\ & \
\sigma (W_i ) \in \Ker ( \varphi) \ \text{and} \ |W_i| > 0 \
\text{for all} \ i \in [0, 2m-2] \}
\end{aligned}
\]
and
\[
\Omega (S) = \Omega = \{ (W_0, \ldots, W_{2m-2}) \in \Omega' \mid
|W_1| = \ldots = |W_{2m-2}| = n \} \,.
\]
The elements $(W_0, \ldots, W_{2m-2}) \in \Omega'(S)$ will be called
\ {\it product decompositions} \ of $S$. If $W\in \Omega'$, we
implicitly assume that $W=(W_0,\ldots,W_{2m-2})$.
\end{definition}
By Lemma \ref{2.5}, $\Omega \ne \emptyset$, and if $W \in \Omega$, then \ $\varphi (W_0), \ldots, \varphi
(W_{2m-2})$ \ are minimal zero-sum sequences over $\varphi (G)$. Proposition \ref{mainproposition} below shows that $\varphi(S)$ is highly structured. We will later in CLAIMS A, B and C of Section 5 (with much effort) show that this structure lifts to the original sequence $S$. As this lift will only be `near perfect' (there will be one exceptional term $x|S$ for which the structure is not shown to lift), we will then, in CLAIM D of Section 5, need Theorem \ref{ham-result} to finish the proof of the Theorem.

\medskip
\begin{proposition} \label{mainproposition}
Let \ $G = C_{mn} \oplus C_{mn}$ \ with \ $m,\, n \ge 2$, and
suppose that   \ $C_n \oplus C_n$ \ has Property {\bf B}. Let \ $S
\in \mathcal A (G)$ with $|S| = 2mn-1$, and  let \ $\varphi \colon G
\to G$ be the multiplication by $m$ homomorphism. Then there exist a
product decomposition $(W_0, \ldots, W_{2m-2})$ of $S$ and a basis
$(e_1, e_2)$ of $\varphi (G)$ such that \be\label{modulo-form}
\varphi (W_0) = e_1^{n-1} \prod_{\nu=1}^n (x_{\nu} e_1 + e_2) \quad
\text{and} \quad \varphi (W_i) \in \bigl\{ e_1^{n}, \
\prod_{\nu=1}^{n} (c_{i, \nu} e_1+e_2) \bigr\} \,, \ee where $x_1,
\ldots , x_n \in [0, n-1]$, $x_1 + \ldots + x_n \equiv 1 \mod n$,
all $c_{i, \nu} \in [0, n-1]$, and
$c_{i,1}+c_{i,2}+\ldots+c_{i,n}\equiv 0\mod n$ for all $i \in
[1,n]$. In particular,
$$
\varphi(S) = e_1^{\ell n-1} \prod_{\nu=1}^{2mn-\ell n} \left( x_{\nu} e_1
+ e_2 \right),$$
where $\ell \in [1, 2m-1]$ and
$x_{\nu} \in [0, n-1]$ for all $\nu \in [1, 2mn-\ell n]$.
\end{proposition}

\begin{proof}
If $n=2$, then it is easy to see (in view of Lemma \ref{2.5}) that
(\ref{modulo-form}) holds. From now on we assume that $n\ge 3$. We
distinguish two cases.

\medskip
\noindent CASE 1: \, For every product decomposition $W \in \Omega$, there exist distinct elements $g_1,\, g_2 \in
\varphi (G)$ such that $\mathsf v_{g_1} \bigl( \varphi (W_0) \bigr)
= \mathsf v_{g_2} \bigl( \varphi (W_0) \bigr) = n-1$.

Let us fix a product decomposition $W \in
\Omega$. By Lemma \ref{2.3}, there is a basis $(e_1, e_2')$ of
$\varphi (G)$ such that
\[
\varphi (W_0) = e_1^{n-1} \prod_{\nu=1}^n (x_{\nu} e_1 + e_2')
\]
where $x_1, \ldots , x_n \in [0, n-1]$ and$x_1 + \ldots + x_n
\equiv 1 \mod n$. Thus, by assumption of CASE 1, it follows that
\[
\varphi (W_0) = e_1^{n-1} (x e_1+e_2')^{n-1} \bigl( (1+x)e_1+e_2'
\bigr) \quad \text{with} \quad x \in [0, n-1] \,.
\]
As a result,
\[
(e_1, e_2) = (e_1, xe_1+e_2') = (e_1, e_2')\cdot
            \left(
            \begin{matrix}
            1 & x \\
            0 & 1
            \end{matrix}
            \right)
\]
is a basis of $\varphi (G)$ and
\[
\varphi (W_0) = e_1^{n-1} e_2^{n-1} (e_1+e_2) \,.
\]
We continue with the following assertion.

\begin{enumerate}
\item[{\bf A.}\,] For every $i \in [1,2m-2]$, $\varphi (W_i)$ has one of the
                  following forms:
                  \[
                  e_1^n, e_2^n, (e_1+e_2)^n, (-e_1+e_2)^n,
                  (e_1-e_2)^n, e_1(e_1+e_2)^{n-2}(e_1+2e_2),
                  e_2(e_1+e_2)^{n-2}(2e_1+e_2) \,.
                  \]
\end{enumerate}

\smallskip
\noindent Suppose that \,{\bf A}\, is proved. If the two forms
$(e_1-e_2)^n$ and $e_1(e_1+e_2)^{n-2}(e_1+2e_2)$ do not occur, then
$\varphi (W_i)$ has the required form with basis $(e_1, e_2)$. If
the two forms $(-e_1+e_2)^n$ and $e_2(e_1+e_2)^{n-2}(2e_1+e_2)$ do
not occur, then $\varphi (W_i)$ has the required form with basis
$(e_2, e_1)$. Thus by symmetry,  it remains to verify that there are
no distinct $i, j \in [1, 2m-2]$ such that
\begin{itemize}
\item[(i)] $\varphi (W_i) = e_1(e_1+e_2)^{n-2}(e_1+2e_2)$ and $\varphi (W_j) =
           e_2(e_1+e_2)^{n-2}(2e_1+e_2)$,

\item[(ii)] $\varphi (W_i) = e_1(e_1+e_2)^{n-2}(e_1+2e_2)$ and $\varphi (W_j) =
            (-e_1+e_2)^n$, or

\item[(iii)] $\varphi (W_i) = (e_1-e_2)^n$ and $\varphi (W_j) = (-e_1+e_2)^n$.
\end{itemize}
Indeed, if (i) held, then
$(2e_1+e_2)(e_1+2e_2)(e_1+e_2)^{n-3}$ would be a zero-sum
subsequence of $\varphi (W_iW_j)$ of length $n-1$, contradicting
Lemma \ref{2.5}. If (ii) held, then
$(-e_1+e_2)(e_1+2e_2)e_2^{n-3}$ would be a zero-sum subsequence of
$\varphi (W_0W_iW_j)$ of length $n-1$, contradicting Lemma
\ref{2.5}. Finally, if (iii) held, then $(e_1-e_2)(-e_1+e_2)$ would
be a zero-sum subsequence of $\varphi (W_iW_j)$ of length $2$, also
contradicting  Lemma \ref{2.5}. Thus it remains to establish \textbf{A} to complete the case.
To that end, let $i \in [1, 2m-2]$ be arbitrary. Then
$\mathsf h \bigl( \varphi(W_0 W_i) \bigr) \ge n-1$, and we
distinguish three subcases.

\medskip
\noindent CASE 1.1: \,$\mathsf h \bigl( \varphi(W_0 W_i) \bigr) >n$.

Then $\mathsf v_g \bigl( \varphi(W_0 W_i) \bigr) > n$ for some $g
\in \{e_1, e_2, e_1+e_2 \}$. If $g = e_1+e_2$, then $\varphi (W_i) =
(e_1+e_2)^n$. Now suppose that $g \in \{e_1, e_2 \}$, say $g = e_1$.
Then
\[
\varphi( W_0 W_i) =  e_2^{n-1} (e_1+e_2) e_1^n \prod_{\nu=1}^{n-1}
(c_{\nu} e_1 + d_{\nu} e_2) \,,
\]
where $c_{\nu}, d_{\nu} \in [0, n-1]$ for all $\nu \in [1, n-1]$, and $c_{\nu}=1$ and $d_{\nu}=0$ for some $\nu\in [1,n-1]$. By
Lemma \ref{2.5},
\[
W_0' =  e_2^{n-1} (e_1+e_2) \prod_{\nu=1}^{n-1} (c_{\nu} e_1 + d_{\nu}
e_2)
\]
is a minimal zero-sum subsequence of $\varphi (S)$. Since $W'$ contains two distinct elements with
multiplicity $n-1$ (by
assumption of CASE 1), and since $e_1|W'_0$, it follows that either
\[
W_0' = e_1^{n-1} e_2^{n-1} (e_1+e_2) \quad \text{or} \quad W_0' = e_1
e_2^{n-1} (e_1+e_2)^{n-1} \,.
\]
But in the second case, we would get $\sigma (W'_0) = - 2e_2 \ne 0$.
Thus $W_0' = e_1^{n-1}e_2^{n-1}(e_1+e_2)$ and $\varphi (W_i) = e_1^n$.

\medskip
\noindent CASE 1.2: \,$\mathsf h \bigl( \varphi(W_0 W_i) \bigr) =
n$.
We distinguish two further subcases.

\medskip
\noindent CASE 1.2.1: \,$\varphi (W_i) = g^n$ for some $g \in
\varphi (G) \setminus \{e_1, e_2, e_1+e_2 \}$.

We set $g = ce_1 + d e_2$ with $c, d \in [0, n-1]$.
 By Lemmas
\ref{2.2} and  \ref{2.5}, it follows that $\varphi (W_0) g^{n-1}$ has a zero
subsequence $T$ of length $|T|=n$ and that $\varphi( W_i W_0) T^{-1}$ is
a minimal zero-sum subsequence of $\varphi (S)$ of length $2n-1$,
say
\[
\varphi (W_i W_0)  T^{-1} = e_2^q e_1^r (e_1+e_2)^s (c e_1+de_2)^t,
\]
where $q\ge 1$, $r\ge 1$, $s\ge 0$ and $t \in [1, n-1]$.

Since $g \ne e_1+e_2$, we infer that $s\le 1.$ If $s=1$, then, by
assumption of CASE 1, we get
\[
2n-1 = |W_i W_0 T^{-1}| =  q+r+s+t\ge 1+(q+r+t)\ge
1+(n-1+n-1+1)>2n-1 \,,
\]
a contradiction. Hence $s=0$.  Again, by assumption of CASE 1, we
have the following possibilities:
\begin{itemize}
\item $q=r=n-1$ and $t=1$.

\item $q=t=n-1$ and $r=1$.

\item $q=1$ and  $r=t=n-1$.
\end{itemize}
If  $q=r=n-1$ and $t=1$, then $\sigma \bigl( \varphi(W_0W_i) T^{-1}
\bigr) = 0$ implies that $g = e_1+e_2$, a contradiction.
If  $q=t=n-1$ and  $r=1$, then $\sigma \bigl( (W_0W_i) T^{-1} \bigr)
= 0$ implies that $g = e_1-e_2$ and $\varphi (W_i) = (e_1-e_2)^n$.
Finally, if  $q=1$ and  $r=t=n-1$, then $\sigma \bigl( \varphi(W_0W_i) T^{-1}
\bigr) = 0$ implies that $g = -e_1+e_2$ and $\varphi (W_i) =
(-e_1+e_2)^n$.

\medskip
\noindent CASE 1.2.2: \,$\mathsf v_g \bigl( \varphi(W_0 W_i) \bigr)
= n$ for some $g \in \{e_1, e_2, e_1+e_2\}$.

Since $|W_i| = n$, $\sigma(\varphi(W_i))=0$ and $\mathsf v_{e_1+e_2} \bigl( \varphi(W_0)
\bigr) = 1$, it follows that $g \ne e_1+e_2$. Thus $g \in \{e_1,
e_2\}$, say $g = e_1$. Then
\[
\varphi (W_0 W_i) = e_2^{n-1} (e_1+e_2) e_1^n \prod_{\nu=1}^{n-1}
(c_{\nu} e_1 + d_{\nu} e_2),
\]
where $c_{\nu}, d_{\nu} \in [0, n-1]$ for all $\nu \in [1, n-1]$. By
Lemma \ref{2.5} and the assumption of CASE 1.2,
\[
W_0' =  e_2^{n-1} (e_1+e_2) \prod_{\nu=1}^{n-1} (c_{\nu} e_1 + d_{\nu}
e_2)
\]
is a minimal zero-sum subsequence of $\varphi (S)$ with $e_1\nmid W'_0$. Since $W'$ contains two distinct elements with
multiplicity $n-1$ (by the
assumption of CASE 1), since $\sigma(\varphi(W_i))=0$, and since $e_1\nmid W'_0$, it follows that
\[
W_0' = e_2^{n-1} (e_1+e_2)^{n-1} (e_1+2e_2),
\]
and thus
\[
\varphi (W_i) = e_1 (e_1+e_2)^{n-2} (e_1+2e_2) \,.
\]

\medskip
\noindent CASE 1.3: \,$\mathsf h \bigl( \varphi( W_0 W_i ) \bigr) =
n-1$.

Since $\sigma(\varphi(W_i))=0$, it follows
$\vp_g(\varphi(W_0W_i))\neq n-1$ for $g\notin \{e_1,e_2,e_1+e_2\}$.
Suppose $\mathsf v_{e_1+e_2}( \varphi(W_0W_i) )  = n-1$. Then
\[
\varphi (W_i) = (e_1+e_2)^{n-2}(c_1e_1+d_1e_1)(c_2e_1+d_2e_2),
\]
where $c_1, d_1, c_2, d_2 \in [0, n-1]$. By Lemmas \ref{2.2} and
\ref{2.5} and the definition of Property \textbf{C},
\[
\varphi (W_0 W_i) (e_1+e_2)^{-1} (c_2e_1+d_2e_2)^{-1}
\]
has a zero-sum subsequence $T$ of length $|T| = n$ and $\varphi (W_0
W_i) T^{-1}$ is a minimal zero-sum subsequence of $\varphi (S)$ of
length $2n-1$. Thus it follows, in view of the assumptions of CASE 1 and CASE 1.3, and in view of
$$\varphi(W_0W_i)=e_1^{n-1}e_2^{n-1}(e_1+e_2)^{n-1}(c_1e_1+d_1e_2)(c_2e_1+d_2e_2),$$ that $\mathsf h (T)=n-1$, contradicting that $\sigma(T)=0$.
So we conclude that

\be\label{piip}\mathsf v_{g}\bigl( \varphi(W_0 W_i) \bigr) < n-1
\quad \text{for all} \quad g \in \varphi (G) \setminus \{e_1, e_2\}
\,. \ee

We set $\varphi (W_i) = \prod_{\nu=1}^n (c_{\nu} e_1 + d_{\nu} e_2)$,
where $c_{\nu}, d_{\nu} \in [0, n-1]$ for all $\nu \in [1,n]$, and
pick some $\lambda \in [1,n]$.  By Lemmas \ref{2.2} and  \ref{2.5}, it follows that
$\varphi (W_0 W_i) (c_{\lambda} e_1 + d_{\lambda} e_2)^{-1}$ has a
zero-sum subsequence $T$ of length $|T| = n$ and that $\varphi (W_i W_0)
T^{-1}$ is a minimal zero-sum subsequence of $\varphi (S)$ of length
$2n-1$. By assumption of CASE 1 and (\ref{piip}), it follows that
\[
\varphi (W_0 W_i) T^{-1} = e_1^{n-1} e_2^{n-1} (e_1+e_2),
\]
and thus $c_{\lambda} e_1 + d_{\lambda} e_2 = e_1+e_2$. As $\lambda\in [1,n]$ was arbitrary, this implies
that $\varphi (W_i) = (e_1+e_2)^n$, contradicting the hypothesis of CASE 1.3.

\medskip
\noindent CASE 2: \, There exists a product decomposition $W \in \Omega$ such that $\mathsf v_g \bigl( \varphi
(W_0) \bigr) = n-1$ for exactly one element $g \in \varphi (G)$.

\medskip
By Lemma \ref{2.3} and the assumption of CASE 2, there exists a basis $(e_1, e_2)$ of $\varphi
(G)$ such that
\[
\varphi (W_0) = e_1^{n-1} \prod_{\nu=1}^n (x_{\nu} e_1 + e_2),
\]
where $x_1, \ldots , x_n \in [0, n-1]$ and $x_1 + \ldots + x_n
\equiv 1 \mod n$ and at most $n-2$ of the elements $x_1, \ldots,
x_n$ are equal. Let $i \in [1, 2m-2]$ be arbitrary, and let $\varphi (W_i) =
\prod_{\nu=1}^n (c_{\nu}e_1 + d_{\nu}e_2)$, where $c_{\nu}, d_{\nu}
\in [0, n-1]$ for all $\nu \in [1,n]$. We proceed to show that there
exists $m_i \in \{0,n\}$ such that
\[
\varphi (W_i) = e_1^{m_i} \prod_{\nu=1}^{n-m_i}(c_{\nu}e_1 + e_2),
\]
which will complete the proof. We  distinguish six subcases.

\medskip
\noindent CASE 2.1: \,$\mathsf h \bigl( \varphi (W_i)
\prod_{\nu=1}^n (x_{\nu}e_1 + e_2) \bigr) > n$.

Then there exists some $x \in [0, n-1]$ such that (after renumbering
if necessary)
\[
\varphi (W_i) \prod_{\nu=1}^n (x_{\nu} e_1 + e_2) = (x e_1 + e_2)^n
\prod_{\nu=1}^r (c_{\nu} e_1 + d_{\nu} e_2) \prod_{\nu=1}^s(x_{\nu}
e_1 + e_2),
\]
where $r \in [1, n-1]$, $s \in [2, n-1]$ and $r+s=n$. Since
\[
e_1^{n-1}\prod_{\nu=1}^r (c_{\nu} e_1 + d_{\nu} e_2)
\prod_{\nu=1}^s(x_{\nu} e_1 + e_2)
\]
is a minimal zero-sum subsequence of $\varphi (S)$, Lemma \ref{2.3}
implies that $d_1 = \ldots = d_r = 1$, whence $\varphi (W_i) =
\prod_{\nu=1}^n (c_{\nu} e_1 + e_2)$.

\medskip
\noindent CASE 2.2: \,$\mathsf h \bigl( \varphi (W_i)
\prod_{\nu=1}^n (x_{\nu}e_1 + e_2) \bigr)=n$.

If $(c_1, d_1) = \ldots = (c_n, d_n)$ does not hold, then, similar
to CASE 2.1, we obtain that $d_1 = \ldots = d_n = 1$. Therefore
$c_1 = \ldots = c_n = c$ and $d_1 = \ldots = d_n = d$ for some
$c,d \in [0, n-1]$.

Pick some $\lambda \in [1, n]$. By Lemmas \ref{2.2} and \ref{2.5}, the definition of Property \textbf{C}, and the assumption of CASE 2,
\[
\varphi (W_0 W_i) (x_{\lambda} e_1 + e_2)^{-1} (c e_1 + de_2)^{-1} =
(c e_1 + de_2)^{n-1} e_1^{n-1} \prod_{\nu \in [1,n] \setminus
\{\lambda\}} (x_{\nu} e_1 + e_2)
\]
has a zero-sum subsequence $T$ of length $n$ and
\[
\varphi (W_0 W_i) T^{-1}
\]
is a minimal zero-sum subsequence of $\varphi (S)$ of length $2n-1$.
Since $\varphi (G)$ has Property {\bf B}, we have either
\[
e_1^{n-1} \t \varphi (W_0 W_i) T^{-1} \quad \text{or} \quad (ce_1 +
de_2)^{n-1} \t \varphi (W_0 W_i) T^{-1} \,.
\]
If $e_1^{n-1} \t \varphi (W_0 W_i) T^{-1}$, then, since
$(x_{\lambda} e_1 + e_2)(ce_1+de_2) \t \varphi (W_0 W_i) T^{-1}$, it
would follow that $d=1$, whence $\varphi (W_i) = (ce_1+e_2)^n$, as desired. Therefore $(ce_1+de_2)^{n-1} \t \varphi (W_0 W_i) T^{-1}$.

Since
$\varphi (W_i)$ is a minimal zero-sum sequence, it follows that
\[
n = \ord (ce_1+de_2) = \frac{n}{\gcd (c, d, n)} \,,
\]
and hence there are $u, v \in \mathbb Z$ such that $uc + vd \equiv 1
\mod n$. Thus
\[
(e_1', e_2') = (ce_1 + de_2, -v e_1+ue_2) = (e_1, e_2)\cdot
            \left(
            \begin{matrix}
            c & -v \\
            d & u
            \end{matrix}
            \right)
\]
is a basis of $\varphi (G)$ and, for some sequence $Q$ over $\varphi
(G)$,
\[
\begin{aligned}
\varphi (W_0 W_i) T^{-1} & = (ce_1+de_2)^{n-1} e_1 (x_{\lambda} e_1
+ e_2)
Q \\
 & = {e_1'}^{n-1} (u e_1' - de_2') \bigl( (x_{\lambda}u+v)e_1' +
 (c-x_{\lambda} d) e_2' \bigr) Q \,.
\end{aligned}
\]
Now Lemma \ref{2.3} implies that $-d \equiv c - x_{\lambda} d\mod n$, whence
$x_{\lambda} d \equiv c+d\mod n$. Therefore, since $\lambda$ was arbitrary, we get
\[
d \equiv \sum_{\nu=1}^n x_{\nu} d \equiv n(c+d) \equiv 0 \mod n,
\]
and thus $d=0$. If $c \in [2,n]$, then $(ce_1)e_1^{n-c}$ is a
zero-sum subsequence of $\varphi (S)$ of length $n-c+1<n$, a
contradiction. Thus $c=1$ and $\varphi (W_i) = e_1^n$.

\medskip
\noindent CASE 2.3: \,$\mathsf h \bigl( \varphi (W_i)
\prod_{\nu=1}^n (x_{\nu}e_1 + e_2) \bigr)=n-1$ and $\mathsf v_{e_1}
\bigl( \varphi (W_i) \bigr)\ge 2.$

After renumbering if necessary, we have
\[
\varphi (W_0 W_i) = e_1^{n+1} (xe_1 + e_2)^{n-1}
\prod_{\nu=1}^r(x_{\nu} e_1 + e_2) \prod_{\nu=1}^s (c_{\nu} e_1 +
d_{\nu} e_2)
\]
where $x \in [0,n-1], r \in [1, n-1], s \in [1, n-2]$ and $r+s =
n-1$. By Lemma \ref{2.5},
\[
W' = e_1 (xe_1 + e_2)^{n-1} \prod_{\nu=1}^r(x_{\nu} e_1 + e_2)
\prod_{\nu=1}^s (c_{\nu} e_1 + d_{\nu} e_2)
\]
is a minimal zero-sum subsequence of $\varphi (S)$ of length $2n-1$.
Since
\[
(e_1, e_2') = (e_1, xe_1+e_2) = (e_1, e_2)\cdot
            \left(
            \begin{matrix}
            1 & x \\
            0 & 1
            \end{matrix}
            \right)
\]
is a basis of $\varphi (G)$ and
\[
W' = e_1 {e_2'}^{n-1} \prod_{\nu=1}^r \bigl( (x_{\nu}-x)e_1 +
e_2'\bigr) \prod_{\nu=1}^s \bigl( (c_{\nu} - x d_{\nu})e_1 + d_{\nu}
e_2' \bigr) \,,
\]
Lemma \ref{2.3} implies that $x_{\nu} - x \equiv 1\mod n$ for all $\nu \in
[1,r]$. Therefore we get $(n-r)x+ r(x+1) \equiv \sum_{\nu=1}^n x_{\nu}
\equiv 1 \mod n$. Hence $r=1$ and
\[
\varphi (W_0) = e_1^{n-1} (x e_1+e_2)^{n-1} \bigl( (x+1)e_1+e_2
\bigr) \,,
\]
a contradiction to our assumption on $x_1, \ldots, x_n$ for CASE 2.

\medskip
\noindent CASE 2.4: \,$\mathsf h \bigl( \varphi (W_i)
\prod_{\nu=1}^n (x_{\nu}e_1 + e_2) \bigr)=n-1$ and $\mathsf v_{e_1}(
W_i )=1.$

After renumbering if necessary, we get
\[
\varphi (W_0 W_i) = e_1^n (xe_1+e_2)^{n-1} \prod_{\nu=1}^r
(x_{\nu}e_1+e_2) \prod_{\nu=1}^s (c_{\nu} e_1 + d_{\nu} e_2)
\]
with $x \in [0, n-1]$, $r \in [1,n-1]$, $s \in [1,n-1]$ and $r+s=n$.
By Lemma \ref{2.5},
\[
W' = (xe_1+e_2)^{n-1} \prod_{\nu=1}^r (x_{\nu}e_1+e_2)
\prod_{\nu=1}^s (c_{\nu} e_1 + d_{\nu} e_2)
\]
is a minimal zero-sum subsequence of $\varphi (S)$ of length $2n-1$.
Since
\[
(e_1, e_2') = (e_1, xe_1+e_2) = (e_1, e_2)\cdot
            \left(
            \begin{matrix}
            1 & x \\
            0 & 1
            \end{matrix}
            \right)
\]
is a basis of $\varphi (G)$ and
\[
W' = {e_2'}^{n-1} \prod_{\nu=1}^r \bigl( (x_{\nu}-x)e_1 + e_2'\bigr)
\prod_{\nu=1}^s \bigl( (c_{\nu} - x d_{\nu})e_1 + d_{\nu} e_2'
\bigr) \,,
\]
Lemma \ref{2.3} implies that
\be\label{starrthing}
x_1-x \equiv \ldots \equiv x_r-x\equiv c_1-xd_1 \equiv \ldots \equiv c_s-xd_s \mod n.\ee
If $d_1 = \ldots = d_s = 1$, then $\varphi (W_i) = \prod_{\nu=1}^n
(c_{\nu}e_1+e_2)$, as desired. Therefore there is some $\nu \in [1,s]$ with
$d_{\nu} \ne 1$, say $\nu = s$. Hence, since $\sigma(W_i)=0$, it follows that there is also another $\nu'\in [1,s]$ with $d_{\nu'}\neq 1$ and $s=\nu\neq \nu'$. Thus, by Lemmas \ref{2.2} and {2.5} and the definition of Property \textbf{C},
\[
\varphi (W_0 W_i) e_1^{-1} (c_se_1+d_se_2)^{-1}
\]
has a zero-sum subsequence $T$ of length $|T| = n$ and $\varphi (W_0
W_i)T^{-1}$ is a minimal zero-sum subsequence of $\varphi (S)$ of
length $2n-1$. Since $\varphi (G)$ has Property {\bf B}, it follows
that either
\[
e_1^{n-1} \t \varphi (W_0 W_i) T^{-1} \quad \text{or} \quad
(xe_1+e_2)^{n-1} \t \varphi (W_0 W_i) T^{-1} \,.
\]
If $e_1^{n-1} \t \varphi (W_0 W_i) T^{-1}$, then,  since
$(c_se_1+d_se_2) \t \varphi (W_0 W_i) T^{-1}$ and $(x_{j}e_1 +
e_2) \t \varphi (W_0 W_i) T^{-1}$ for some $j \in [1,n]$, Lemma
\ref{2.3} implies that $d_s=1$, a contradiction. Therefore
$(xe_1+e_2)^{n-1} \t \varphi (W_0 W_i) T^{-1}$. Thus, for some
sequence $Q$ over $\varphi (G)$, we have
\[
\varphi (W_0 W_i)T^{-1} = (xe_1+e_2)^{n-1} e_1 (c_se_1+d_se_2)Q \,.
\]
Since
\[
(e_1, e_2') = (e_1, xe_1+e_2) = (e_1, e_2)\cdot
            \left(
            \begin{matrix}
            1 & x \\
            0 & 1
            \end{matrix}
            \right)
\]
is a basis of $\varphi (G)$ and
\[
\varphi (W_0 W_i)T^{-1} = e_1 {e_2'}^{n-1} \bigl( (c_s-xd_s)e_1 +
d_se_2' \bigr) Q \,,
\]
Lemma \ref{2.3} implies that $c_s - xd_s = 1$. Thus it follows from
$(\ref{starrthing})$ that $x_1 \equiv \ldots \equiv x_r \equiv x+1\mod n$. Therefore we get $(n-r)x+
r(x+1) \equiv \sum_{\nu=1}^n x_{\nu} \equiv 1 \mod n$. Hence $r=1$ and
\[
\varphi (W_0) = e_1^{n-1} (x e_1+e_2)^{n-1} \bigl( (x+1)e_1+e_2
\bigr) \,,
\]
a contradiction to our assumption on $x_1, \ldots, x_n$ for CASE 2.

\medskip
\noindent CASE 2.5: \,$\mathsf h \bigl( \varphi (W_i)
\prod_{\nu=1}^n (x_{\nu}e_1 + e_2) \bigr)=n-1$ and $\mathsf v_{e_1}
\bigl( \varphi (W_i ) \bigr) = 0$.

If $d_1= \ldots = d_n=1$, then the assertion follows. Therefore there
is some $\nu \in [1,n]$ with $d_{\nu} \ne 1$, say $\nu = n$. Since
$d_1 + \ldots + d_n \equiv 0 \mod n$, we may also assume that
$d_{n-1} \ne 1$. We distinguish two subcases.

\medskip
\noindent CASE 2.5.1: \,$\varphi (W_i) \prod_{\nu=1}^n (x_{\nu}e_1 +
e_2)$ contains two distinct elements with multiplicity $n-1$, say
$xe_1+e_2$ and $ye_1+e_2$, where $x,\,y \in [0, n-1]$.

Then
\[
\varphi (W_i) =
(xe_1+e_2)^r(ye_1+e_2)^s(c_{n-1}e_1+d_{n-1}e_2)(c_ne_1+d_ne_2)
\]
and
\[
\prod_{\nu=1}^n (x_{\nu}e_1 + e_2) =
(xe_1+e_2)^{n-1-r}(ye_1+e_2)^{n-1-s},
\]
where $r,\,s \in [1,n-3]$ and $r+s=n-2\ge2$. By Lemmas \ref{2.2} and
\ref{2.5}, $\varphi (W_0 W_i) (c_ne_1+d_ne_2)^{-1}$ has a zero-sum
subsequence $T$ of length $|T| = n$ and $\varphi (W_0 W_i)T^{-1}$ is
a minimal zero-sum subsequence of $\varphi (S)$ of length $2n-1$.
Since $\varphi (G)$ has Property {\bf B}, it follows that
\[
\mathsf v_g \bigl( \varphi( W_i W_0)T^{-1} \bigr) = n-1 \quad
\text{for some} \quad g \in \{e_1, xe_1+e_2, ye_1+e_2\} \,.
\]
Clearly, we have
\[
e_1 (xe_1+e_2)(ye_1+e_2)(c_ne_1+d_ne_2) \t \varphi (W_0 W_i) T^{-1}
\,.
\]
Since $d_n \ne 1$, Lemma \ref{2.3} implies that $g \ne e_1$. Thus w.l.o.g. $g = xe_1+e_2$. Consequently, for some
sequence $Q$ over $\varphi (G)$, we have
\[
\varphi (W_0 W_i) T^{-1} = (xe_1+e_2)^{n-1}e_1 (ye_1+e_2) Q \,.
\]
As before,
\[
(e_1, e_2') = (e_1, xe_1+e_2) = (e_1, e_2)\cdot
            \left(
            \begin{matrix}
            1 & x \\
            0 & 1
            \end{matrix}
            \right)
\]
is a basis of $\varphi (G)$ and
\[
\varphi (W_0 W_i)T^{-1} = {e_2'}^{n-1}e_1 \bigl( (y-x)e_1+e_2') Q
\,.
\]
Now we obtain a contradiction as in CASE 2.3.

\medskip
\noindent CASE 2.5.2: \,$\varphi (W_i) \prod_{\nu=1}^n (x_{\nu}e_1 +
e_2)$ contains exactly one element with multiplicity $n-1$, say
$xe_1+e_2$ where $x \in [0, n-1]$.

After renumbering if necessary, we get
\[
\varphi (W_0 W_i) = e_1^{n-1}(xe_1+e_2)^{n-1}
\prod_{\nu=1}^r(c_{\nu}e_1+d_{\nu}e_2)
\prod_{\nu=1}^s(x_{\nu}e_1+e_2),
\]
where $r \in [1, n-1], s \in [2, n-1]$ and $r+s=n+1$. If $d_1 =
\ldots = d_r = 1$, then the assertion follows. So after renumbering
again, we suppose that $d_r \ne 1$. Let $\lambda \in
[1,s]$.

By Lemmas \ref{2.2} and \ref{2.5}, the definition of Property \textbf{C}, and the assumption of CASE 2.5.2,
\[
\varphi (W_0 W_i) (c_re_1+d_re_2)^{-1} (x_{\lambda}e_1+e_2)^{-1}
\]
has a zero-sum subsequence $T$ of length $|T| = n$ and $\varphi (W_0
W_i)T^{-1}$ is a minimal zero-sum subsequence of $\varphi (S)$ of
length $2n-1$. Since $\varphi (G)$ has Property {\bf B}, it follows
that
\[
\mathsf v_g \bigl( \varphi (W_0 W_i)T^{-1} \bigr) = n-1 \quad
\text{for some} \quad g \in \{e_1, xe_1+e_2, \} \,.
\]
Clearly, we have
\[
e_1 (xe_1+e_2)(c_re_1+d_re_2)(x_{\lambda}e_1+e_2) \t \varphi (W_0
W_i) T^{-1} \,.
\]
Since $d_r \ne 1$, Lemma \ref{2.3} implies that $g \ne e_1$, and
hence $g = xe_1+e_2$. Thus, for some sequence $Q$ over $\varphi
(G)$, we have
\[
\varphi (W_0 W_i) T^{-1} = (xe_1+e_2)^{n-1}e_1(x_{\lambda}e_1+e_2)Q
\,.
\]
As before,
\[
(e_1, e_2') = (e_1, xe_1+e_2) = (e_1, e_2)\cdot
            \left(
            \begin{matrix}
            1 & x \\
            0 & 1
            \end{matrix}
            \right)
\]
is a basis of $\varphi (G)$ and
\[
\varphi (W_0 W_i) T^{-1} = {e_2'}^{n-1}e_1 \bigl(
(x_{\lambda}-x)e_1+e_2') Q \,.
\]
Hence Lemma \ref{2.3} implies that $1 \equiv x_{\lambda}-x\mod n$.
As $\lambda\in [1,s]$ was arbitrary, it follows that $x_1\equiv
\ldots \equiv x_s \equiv x+1 \mod n$, and, as in CASE 2.3, we obtain
a contradiction.

\medskip
\noindent CASE 2.6: \,$\mathsf h \bigl( \varphi (W_i)
\prod_{\nu=1}^n (x_{\nu}e_1 + e_2) \bigr) < n-1$.

Let $\lambda \in [1, n]$ be arbitrary.
By
Lemmas \ref{2.2} and \ref{2.5},
\[
\varphi (W_0 W_i) (c_{\lambda} e_1 + d_{\lambda}e_2)^{-1}
\]
has a zero-sum subsequence $T$ of length $|T| = n$, and
\[
\varphi (W_0 W_i) T^{-1}
\]
is a minimal zero-sum subsequence of $\varphi (S)$ of length $2n-1$.
Since $\varphi (G)$ has Property {\bf B}, it follows that
$e_1^{n-1}$ divides $\varphi (W_0 W_i) T^{-1}$. Furthermore there is
some $\nu \in [1,n]$ such that
\[
(x_{\nu} e_1 + e_2)(c_{\lambda} e_1 + d_{\lambda} e_2)|\varphi (W_0 W_i) T^{-1} \,.
\]
Thus Lemma \ref{2.5} implies that either $d_{\lambda} = 1$ or
$(c_{\lambda}, d_{\lambda}) = (1,0)$. Thus, since $\lambda\in [1,n]$ was arbitrary and $\sigma(\varphi(W_i))=0$, we must either have $d_\lambda=1$ for all $\lambda\in [1, n]$ or  $(c_{\lambda},
d_{\lambda}) = (1,0)$ for all $\lambda\in [1, n]$, and so either $\varphi (W_i) = e_1^n$ or $\varphi(W_i)=\prod_{\nu=1}^{n}(c_{\nu}e_1+e_2)$, as desired
\end{proof}

\bigskip
\section{Proof of the Theorem} \label{5}
\bigskip

Let $G = C_{mn} \oplus C_{mn}$, with $m,\, n \geq 3$ odd, $mn > 9$
and w.l.o.g. $m\geq 5$, such that Property {\bf B} holds both for
$C_m \oplus C_m$ and $C_n \oplus C_n$. Let $S\in {\mathcal A} (G)$
be a minimal zero-sum sequence of length $|S|=2mn-1$. The
sequence $S$ will remain fixed throughout the rest of this section.
Our goal is to show that $S$ contains an element with multiplicity
$mn-1$ (in other words, $\mathsf h (S) = mn-1$). We proceed in the
following way\,{\rm :}
\begin{itemize}
\smallskip
\item First, using Proposition \ref{mainproposition}, we establish the setting and some detailed notation necessary
to formulate the key ideas of the proof.

\smallskip
\item Next, we proceed with four lemmas, Lemmas \ref{Step-cant-flip-badly},
      \ref{lemma-for-two-opts}, \ref{lem-newstuff} and \ref{lem-aligment}, that collect
      several arguments used repeatedly in the proof.

\smallskip
\item Then we divide the main part of the proof into four claims, CLAIMS A, B, C and D, where in CLAIM D we finally
      show that $\mathsf h (S) = mn-1$.
\end{itemize}

\medskip
\centerline{\bf The Setting and Key Definitions}
\medskip

Since $S$ is fixed, we write $\Omega'$ and $\Omega$ instead of
$\Omega' (S)$ and $\Omega (S)$ (see Definition \ref{4.1}). Recall
that Lemma \ref{2.3}.3 implies that $\ord (x) = mn$ for all $x \in
\supp (S)$. Let $\varphi \colon G \to G$ denote the multiplication
by $m$ map. Then \ $\Ker ( \varphi) = nG \cong C_m \oplus C_m$ \ and
\ $\varphi (G) = mG \cong C_n \oplus C_n$.

Let $\Omega_0 \subset \Omega$ be all those $W\in \Omega$ for which
there exists a basis $(me_1,me_2)$ of $\varphi(G)$, where $e_1, e_2
\in G$, such that
$\varphi(W_{0})=(me_1)^{n-1}\prod_{\nu=1}^n(x_{\nu}me_1+me_2)$,
where  $x_1, \ldots, x_n  \in \Z$ with $x_1 + \ldots + x_n \equiv
1\mod n$, and such that for every  $i \in [1, 2m-2]$, $\varphi(W_i)$
is either of the form $\varphi(W_i)=(me_1)^{n}$, or of the form
$\varphi(W_i)=\prod_{\nu=1}^n (y_{i,\nu}me_1+me_2)$, where $y_{i,1},
\ldots, y_{i,n} \in \Z$ with $y_{i,1} + \ldots + y_{i,n} \equiv
0\mod n$. By Proposition \ref{mainproposition},  $\Omega_0 $ is
nonempty.

Let $W \in \Omega'$, and define $\widetilde{\sigma} (W) =
\prod_{\nu=0}^{2m-2}\sigma(W_{\nu}) \in \Fc \bigl( \Ker(\varphi)
\bigr)$. Since $S\in \A(G)$, it follows that $\widetilde{\sigma}
(W)\in \A \bigl( \Ker(\varphi) \bigr)$. Thus, since Property {\bf B}
holds for $\Ker (\varphi)$, it follows that $\widetilde{\sigma}
(W)\in \Upsilon \bigl( \Ker(\varphi) \bigr)$. Partition $\Omega_0 =
\Omega^u_0 \cup \Omega^{nu}_0 $ by letting $\Omega^u_0 $ be those
$W\in \Omega_0$ with $\widetilde{\sigma} (W)\in
\Upsilon_{u}\bigl(\Ker(\varphi) \bigr)$, and letting $\Omega^{nu}_0
$ be those $W\in \Omega_0$ with $\widetilde{\sigma} (W)\in
\Upsilon_{nu} \bigl(\Ker(\varphi) \bigr)$.

Let $W  \in \Omega_0 $, let $(me_1,me_2)$  be a basis of
$\varphi(G)$ satisfying the definition of $\Omega_0 $, with
$e_1,\,e_2\in G$, and let $(f_1,f_2)$ be a basis for $\Ker(\varphi)$
such that $\widetilde{\sigma} (W)$ can be written as in  the
definition of $\Upsilon \bigl( \Ker ( \varphi) \bigr)$. Let $S_1$ be
the subsequence of $S$ consisting of all terms $x$ with $\varphi(x)
= me_1$, and define $S_2$ by $S=S_1S_2$. Let $I\subset \Z$ be an
interval of length $n$. Then each term $x$ of $S_1$ has a unique
representation of the form $x=e_1 + ng$, with $n
g\in \Ker(\varphi)$ (where $g \in G$), and each term $x$ of $S_2$ has a unique
representation of the form $x=a e_1+e_2+ng$, with $a \in I$ and $ng\in \Ker(\varphi)$ (where $g \in
G$). Define $\psi(x)= n g \in
\Ker(\varphi)$ and, for $x \in \supp (S_2)$, define $\iota(x)=a \in
I\subset \Z$. We set $\psi (x) = \psi_1 (x) + \psi_2 (x)$, where
$\psi_1 (x) \in \langle f_1 \rangle$ and $\psi_2 (x) \in \langle f_2
\rangle$. If $y\in \Ker(\varphi)$, with $y=y_1f_1+y_2f_2$, then we
also use $\psi_i(y)$ to denote $y_if_i$. Note that, for $x \in \supp
(S_1)$, the value of $\psi(x)$ depends upon the choice of
$(e_1,e_2)$, and that, for $x \in \supp (S_2)$, the values of
$\psi(x)$ and $\iota(x)$ depend upon the choice of $(e_1,e_2)$ and
$I$. We will frequently need to vary the underlying choices for
$(e_1,e_2)$ and $I$, and each time we do so the corresponding values
of $\psi$ and $\iota$ will be affected. All maps will be extended to
sequences as explained before Definition \ref{2.1}.

Let $\A_1(W)$ be those $W_i$ either with $i=0$ or $\varphi(W_i)=
(me_1)^n$, let $\A_2(W)$ be all remaining $W_i$ as well as $W_0$,
and let $\A^*_i(W)=\A_i (W) \setminus \{W_0\}$ for $i \in \{1,2\}$.
If $W \in \Omega_0^u$, let $\C_0(W)$ be all those $W_i$ with
$\mathsf v_{\sigma(W_i)} \bigl( \widetilde{\sigma} (W) \bigr) <
m-1$, let $\C_1(W)$ be all remaining $W_i$, and let $\C^*_i
(W)=\C_i(W)\setminus\{W_0\}$ for $i \in \{0, 1\}$. If $W \in
\Omega_0^{nu}$, let $\C_0(W)$ be the unique $W_i$ with $\mathsf
v_{\sigma(W_i)} \bigl( \widetilde{\sigma}(W) \bigr) < m-1$, and
divide the remaining $2m-2$ blocks $W_i$ into either $\C_1(W)$ or
$\C_2(W)$ depending on the value of $\sigma(W_i)$; analogously
define $\C^*_i(W)$ for $i \in \{0,1,2\}$. When the context is clear, the
$W$ will be omitted from the notation. We regard the elements
$W_i,\,W_j\in\A_1$ as distinct when $i\neq j$, follow the same
convention for all other similar collections of $W_i$, and will
refer to them as blocks.

We further subdivide $W_0=W^{(1)}_0W^{(2)}_0$ with
$W^{(1)}_0=\gcd(W_0,S_1)$ and $W^{(2)}_0=\gcd(W_0,S_2)$, and for a
pair of subsequences $X$ and $Y$ with $XY \t S_2$, we define
$\epsilon'(X,Y)$ to be the integer in $[1,n]$ congruent to
$\sigma(\iota(X))-\sigma(\iota(Y))$ modulo $n$,
and define $\epsilon(X,Y)$ to be the integer such that
$$n-\epsilon'(X,Y)+\sigma(\iota(X))-\sigma(\iota(Y))=\epsilon(X,Y)n.$$

The main idea of the proof is to swap individual terms contained in
the blocks of $W\in \Omega_0$ in such a way so as to maintain that
the resulting product decomposition still lies in $\Omega'$. Using
the lemmas from Section \ref{3}, we will then derive information
about the possible values of $\psi$ and $\iota$ obtained on the terms that have
been swapped. The next three paragraphs detail the three major types
of swaps that we will use.

If $U,\,V\in \A_1$ are distinct (thus $U=W_i$ and $V=W_j$ for some
$i$ and $j$ distinct), then we may exchange any  subsequence $X|U$
for a subsequence $Y|V$ with $|Y|=|X|$ (if $U=W_0$, then $X$ must
additionally lie within $W^{(1)}_0$, and likewise for $V$) and the
resulting product decomposition $W'$ will still lie in $\Omega_0$,
equal to $W$ except that the blocks $U$ and $V$ of $W$ have been
replaced by the blocks $U':=X^{-1}UY$ and $V':=Y^{-1}VX$. Moreover,
\be\label{exchange-setup-I}\sigma(V')=\sigma(V)+
\sigma(\psi(X))-\sigma(\psi(Y)).\ee
We refer to this as a \ {\it type I swap}.

If $V\in \A_2^*$, and $Y|V$ and $X|W^{(2)}_0$ are subsequences with
$|X|=|Y|$, then by exchanging the sequence $Y|V$ for the sequence
$RX|W_0$, where $R|W^{(1)}_0$ is any subsequence with
$|R|=n-\epsilon'(X,Y)$, we obtain a product decomposition $W'$ that
still lies in $\Omega'$, equal to $W$ except that the blocks $V$ and
$W_0$ of $W$ have been replaced by the blocks $V':=Y^{-1}VXR$ and
$W'_0:=R^{-1}X^{-1}W_0Y$. Moreover,
\be\label{exchange-setup-II}\sigma(V')=\sigma(V)+\epsilon(X,Y)ne_1
+\sigma(\psi(X))-\sigma(\psi(Y))+\sigma(\psi(R)).\ee
We refer to this as a \ {\it type II swap}.

If $U,\,V\in \A_2$ are distinct, then we may exchange any
subsequence $X|U$ for a subsequence $Y|V$ with $|Y|=|X|$  and
$\sigma (\iota (X)) = \sigma (\iota (Y))$ (and if $U=W_0$, then $X$
must additionally lie within $W^{(2)}_0$, and likewise for $V$) and
the resulting product decomposition $W'$ will still lie in
$\Omega_0$, equal to $W$ except that the blocks $U$ and $V$ of $W$
have been replaced by the blocks $U':=X^{-1}UY$ and $V':=Y^{-1}VX$.
Moreover, \be\label{exchange-setup-III-0}\sigma(V')=\sigma(V)
+\sigma(\psi(X))-\sigma(\psi(Y)).\ee We refer to this as a \ {\it
type III swap}.

We will often also have need to change from $W \in \Omega_0$ to
another $W'\in \Omega_0$. One common way that this will be done will
be to find $U\in \A^*_2$ and $X|U\Wo$ ($X$ will often be a single
element dividing $U$). Then $|X^{-1}U\Wo|=2n-|X|$. If there is an
$n$-term subsequence $U'|X^{-1}U\Wo$ with $\sigma(U')\in \Ker
(\varphi)$ (as is guaranteed by Theorem \ref{EGZ-thm}.1
in case $|X|=1$), then, defining $W'_0$ by $W'_0U'=W_0U$, we obtain a
new product decomposition $W'\in \Omega_0$ by replacing the blocks
$W_0$ and $U$ by $W'_0$ and $U'$. Moreover, $X \t {W_0'}^{(2)}$. We
refer to such a procedure as \ {\it pulling $X$ up into the new
product decomposition $W'$}.

All of the above procedures result in a new product decomposition $W'\in \Omega'$,
and we will always assume $W'=(W'_0,\ldots,W'_{2m-2})$, with $W'_k=W_k$ for all blocks $W_k$ not involved in the procedure,
and with $W'_i$ and $W'_j$ defined as above for the two blocks $W_i$ and $W_j$ involved in the procedure.

\medskip
\centerline{\bf Four Lemmas}
\medskip

We will often only consider $W\in \Omega_0^{nu}$ when
$\Omega_0^{u}=\emptyset$ (with one exception in CASE 3 of CLAIM C).
The reason for this is to ensure that, if a swapping procedure
applied to $W$ results in a new product decomposition $W'\in
\Omega_0$, then $W'\in \Omega_0^{nu}$ is guaranteed, and hence the
more powerful Lemma \ref{lem-Pertebation-III} is available (instead
of the weaker Lemma \ref{lem-Pertebation-II}).

The following lemma will be used in CASE 3 of CLAIM C to avoid having to consider a $W''\in \Omega_0^{nu}$ when $\Omega_0^{u}\neq \emptyset$.

\begin{lemma} \label{Step-cant-flip-badly}
Let $W\in \Omega_0^u $, $U\in \C_1$ and $V_1,\,V_2\in \C_0$ be
distinct. Suppose there exist $X|U$ and $Y_1|V_1$ such that swapping
$X$ for $Y_1$ yields a new product decomposition $W'\in \Omega'$
with the new block $U'=X^{-1}UY_1$ in $W'$ having $\sigma(U')\neq
\sigma(U)$. If $Y_2|Y_1^{-1}V_1$ and $Z|V_2$ are nontrivial
subsequences such that swapping $Y_2$ for $Z$ in $W$ yields a new
product decomposition $W''\in \Omega_0$, then $W''\in \Omega_0^u$.
\end{lemma}

\begin{proof}
Assume by contradiction that $W''\in \Omega_0^{nu}$, so that
w.l.o.g. $\widetilde{\sigma} (W'')=f_1^{m-1}f_2^{m-1}(f_1+f_2)$
with $\sigma(U)=f_1$ (since $\sigma(U)$ is a maximal multiplicity term in $\sig(W)$ and all blocks involved in the swap resulting in $W''$ are not of maximal multiplicity, it follows that $\sigma(U'')=\sigma(U)$ must be a maximal multiplicity term in $\sig(W'')$ as well). Since $m\geq 4$ (so that $f_2f_2^{m-1}|\sig(W)$), let $\sigma(V_1)=Cf_1+f_2$ with $C \in [0,
m-1]$. By hypothesis, we may swap $Y_1|V''_1=Y_2^{-1}V_1Z$ for $X|U''=U$
to obtain a new product decomposition $W'''\in \Omega'$, with new
respective terms $V'''_1$ and $U'''$. Since (by hypothesis) swapping
$X$ for $Y_1$ in $W$ yields a new product decomposition $W'\in
\Omega'$ such that the new block $U'=X^{-1}UY_1$ in $W'$ has
$\sigma(U')\neq \sigma(U)$, it follows from Lemma
\ref{lem-Pertebation-I}.2 that $\sigma(U''')=\sigma(U')=Cf_1+f_2$
and $\sigma(V'''_1)=\sigma(V''_1)+(1-C)f_1-f_2$.

Suppose $\sigma(V''_1)=f_2$. Then, from the above paragraph, we
conclude that
\[
\widetilde{\sigma}
(W''')=f_2^{m-2}(f_1+f_2)((1-C)f_1)f_1^{m-2}(Cf_1+f_2) \,.
\]
Thus, since $\widetilde{\sigma} (W''')\in \Upsilon(\Ker(\varphi))$
and $m\geq 4$, it follows that $C=0$, whence
$\sigma(V''_1)=f_2=Cf_1+f_2=\sigma(V_1)$. However, this implies that
$\widetilde{\sigma} (W) = \widetilde{\sigma} (W'')\in
\Upsilon^{nu}_0$, contrary to $W\in\Omega_0^u$. So we may assume instead
that $\sigma(V''_1)=f_1+f_2$ (note $\sigma(V''_1)\neq f_1$, since $\sigma(U)=f_1$, $U\in \C_1(W)$ and no terms from $\C_1(W)$ were involved in the swap resulting in $W''$).

In this case, we instead conclude that
$$\widetilde{\sigma} (W''')=f_2^{m-1}((2-C)f_1)f_1^{m-2}(Cf_1+f_2).$$ Thus, since
$\widetilde{\sigma} (W''')\in \Upsilon(\Ker(\varphi))$ and $m\geq
3$, we conclude that $C=1=2-C$, and once more
$\sigma(V''_1)=\sigma(V_1)$, yielding the same contradiction as in
the previous paragraph, completing the lemma.
\end{proof}

\medskip

The next two lemmas will often be used in conjunction, and will form one of our main swapping strategy arguments used for CLAIMS A and B. Note that Lemma \ref{lemma-for-two-opts}(i) gives a strong structural description as well as a term of multiplicity at least $(|\Dc_1|+1)n-1$ in $S$, while Lemma \ref{lemma-for-two-opts}(ii) allows us to invoke Lemma \ref{lem-newstuff}.

\begin{lemma}\label{lemma-for-two-opts}
Let $W\in \Omega_0$ and, if $\Omega^u_0 \neq \emptyset$, assume
that $W\in \Omega^u_0$. Let $\Dc_1,\,\Dc_2\subset \A_2^*$ be such
that, for each (relevant) $i\in [0,2]$, there do not exist $U\in \Dc_1$ and
$V\in \Dc_2$ with $U,\,V\in \C_i$. If either
\begin{enumerate}
\item[(a)] $|\Dc_1|\geq 1$ and every type III swap between $x|\Wo$ and
           $y|W_j$, with $W_j\in \Dc_1$ and
           $\iota(x)=\iota(y)$, results in a new product decomposition $W'$
           with $\sigma(W'_0)=\sigma(W_0)$, \ or

\item[(b)] $|\Dc_1|\geq 2$ and $|\Dc_{2}|\geq 1$,
\end{enumerate}
then one of the following two statements hold{\rm \,:}
\begin{enumerate}
\item[(i)] There exist $x_0|\Wo$, $g\in I$ and $\alpha\in \Ker(\varphi)$
           such that $\iota(x_0)\equiv g+1\mod n$, $\iota(x)=g$ and
           $\psi(x)=\alpha$, for all $x|x_0^{-1}\Wo\prod_{V\in
           \Dc_1}V$.

\item[(ii)] There exist $W_j\in \Dc_1$, $X|\Wo$ and $Y|W_j$ such that
             $|X|=|Y|$ and $\epsilon'(X,Y)\notin \{1,n\}$.
\end{enumerate}
\end{lemma}

\begin{proof}
We assume that (ii) fails and show that (i)  holds. If $W_0\in
\C_0$, then choose $f_2$ such that $\sigma(W_0)=f_1+f_2$; if $W\in
\Omega_0^{nu}$, then choose $f_2$ such that
$\sig(W)=f_1^{m-1}f_2^{m-1}(f_1+f_2)$ (note, in case $W_0\in C_0$
and $W\in \Omega_0^{nu}$, that this choice of $f_2$ agrees with the
previous choice), and assume $\C_1$ consists of those $W_i$ with
$\sigma(W_i)=f_1$; and if $W_0\notin \C_0$, then w.l.o.g. assume
$W_0\in \C_1$.

Applying Lemma \ref{lem-basic-exchange}.3 to $\iota(\Wo)$ and each
$\iota(V)$ with $V\in\Dc_1$, with both sequences considered modulo
$n$ (since (ii) fails, the hypothesis of Lemma
\ref{lem-basic-exchange}.3 holds with $\{0,a\}$ equal to  $\{n, 1\}$
modulo $n$), we conclude, in view of $\sigma(\iota(\Wo))\equiv 1\mod
n$ (and hence $|\supp(\iota(\Wo))|>1$), that there exist $x_0|\Wo$ and $g\in I$ such that
$\iota(x_0)\equiv g+1\mod n$ and $\iota(x)=g$ for all
$x|x_0^{-1}\Wo\prod_{V\in \Dc_1}V$. If (a) holds, then performing type
III swaps between $W_0$ and the $V\in \Dc_1$ completes the proof.
Therefore assume (a) fails and (b) holds instead.

\smallskip

\noindent CASE 1: \  $W_0\in \C_0$.

Thus, since $|\Dc_1|,\,|\Dc_2|\geq 1$, let $U\in \A_2^*\cap (\Dc_1\cup
\Dc_2)$ with $\sigma(U)=f_1$ and let $V\in \A^*_2\cap(\Dc_1\cup
\Dc_2)$ with $\sigma(V)=Cf_1+f_2$ for some $C\in\Z$. Performing a
type II swap between some fixed $u|U$ and each $x|x_0^{-1}\Wo$
(using the same fixed subsequence $R|W_0^{(1)}$ in every swap, which
is possible since $\iota(x)=g$ for all $x|x_0^{-1}\Wo$), we conclude
from either Lemma \ref{lem-Pertebation-I}.2 (since
$\sigma(W_0)=f_1+f_2$) or Lemma \ref{lem-Pertebation-II}.4 that
$\psi_1$ is constant on $x_0^{-1}\Wo$. Likewise performing a type II
swap between some fixed $v|V$ and each $x|x_0^{-1}\Wo$, we conclude
from either Lemma \ref{lem-Pertebation-I}.3 or Lemma
\ref{lem-Pertebation-II}.5 that $\psi_2$ is constant on
$x_0^{-1}\Wo$. Consequently, $\psi(x)=\alpha$ (say) for all
$x|x_0^{-1}\Wo$.

Suppose $W\in\Omega_0^{nu}$. Then $\Dc_1\subset\A_2^*\cap \C_i$, for
some $i\in\{1,2\}$ (in view of the hypotheses of CASE 1 and the lemma), and performing type III swaps between the $Z\in
\Dc_1$, we conclude, in view of $|\Dc_1|\geq 2$ and Lemma
\ref{lem-Pertebation-III}.1 or \ref{lem-Pertebation-III}.2, that
$\psi(x)=\alpha'$ (say) for all $x|\prod_{V\in \Dc_1}V$. Further
applying type III swaps between $W_0$ and any $Z\in \Dc_1$, we
conclude from Lemma \ref{lem-basic-exchange}.3 and either Lemma
\ref{lem-Pertebation-III}.4 or \ref{lem-Pertebation-III}.5 that
$\alpha=\alpha'$, completing the proof. So we may assume
$W\in\Omega_0^{u}$.

If $\Dc_1\subset \C_1$, then repeating the argument of the previous
paragraph using Lemma \ref{lem-Pertebation-I} in place of Lemma
\ref{lem-Pertebation-III} completes the proof. Therefore we may
assume $\Dc_1\subset \C_0$. Let $Z\in \Dc_1$ and $z|Z$. We proceed
to show $\psi(z)=\alpha$, which, since $z|Z\in \Dc_1$ is
arbitrary, will complete the proof.

If performing a type III swap between $z|Z$ and some $x|x_0^{-1}\Wo$
results in a new product decomposition $W'\in \Omega^u_0$, then $W'_0\in \C_0$ (as both $W_0,\,Z\in \C_0$) and,
repeating the arguments of the first paragraph of CASE 1 this time for $W'$, we
conclude that $\psi(z)=\alpha$. If $W'\in \Omega_0^{nu}$, then we
can choose a new $f_2$ such that $\sig(W')=f_1^{m-1}f_2^{m-1}(f_1+f_2)$.
If also $W'_0\in \C_0$, then $\sigma(W'_0)=f_1+f_2$, and  repeating
the arguments of the first paragraph for $W'$ shows $\psi(z)=\alpha$.
Therefore suppose $W'\in \Omega_0^{nu}$ and $\sigma(W'_0)=f_2$. In
view of Lemma \ref{lem-Pertebation-I}.3, we have $\alpha-\psi(z)\in
\langle f_1 \rangle$. However, if $\alpha\neq \psi(z)$, then
performing a type II swap between some $y|U'=U$ and both $z|W'_0$
and $z'|W'_0$, where $\iota(z')=g$ and $\psi(z')=\alpha$, we
conclude from Lemma \ref{lem-Pertebation-II}.3 that $$\epsilon ne_1
+\sigma(\psi(R))-\psi(y)+\{\psi(z),\alpha\}=\{0,f_2-f_1\},$$ where
$\epsilon=\epsilon(z,y)=\epsilon(z',y)$ (in view of
$\iota(z)=\iota(z')=g$) and $R$ is the same fixed subsequence of
${W_0'}^{(1)}$ used in both swaps (also possible since
$\iota(z)=\iota(z')=g$). Hence $\psi(z)-\alpha=\pm(f_2-f_1)$,
contradicting that $\alpha-\psi(z)\in \langle f_1 \rangle$, and
completing CASE 1.

\smallskip

\noindent CASE 2: \  $W_0\notin C_0$ \ and \ $W\in \Omega_0^{nu}$.

Then $W_0\in \C_1$ (by our normalizing assumptions). If there is $Z\in \Dc_1\cap \C_0$ and $\Dc_1\cap \C_2=\emptyset$, then, in view of Lemma \ref{lem-Pertebation-III}.4, we may assume that performing any type III swap between $z|Z$ and $x|x_0^{-1}\Wo$ results in a product decomposition $W'$ with $\sigma(W'_0)=\sigma(W_0)$, else CASE 1 applied to $W'$ completes the proof. Note that Lemma \ref{lem-Pertebation-III}.1 guarantees the same for any $Z\in \Dc_1\cap \C_1$. Thus if $\Dc_1\cap \C_2\neq \emptyset$, then (a) holds, contrary to assumption, and so we may assume instead that $\Dc_1\cap C_2\neq \emptyset$.

Suppose there is $Z\in \Dc_2$ with $\sigma(Z)=f_1+f_2$. Then performing type II swaps between some $z|Z$ and each $x|x_0^{-1}\Wo$ (using the same $R|W_0^{(1)}$ for every swap, which is possible since $\iota(x)=g$ for all $x|x_0^{-1}\Wo$), we conclude from Lemma \ref{lem-Pertebation-II}.4 that $\psi_1$ is constant on $x_0^{-1}\Wo$. If we perform type III swaps between $U$ and $W_0$ with $U\in\Dc_1\cap \C_2$, then we conclude from Lemmas \ref{lem-Pertebation-II}.3 and \ref{lem-basic-exchange}.3 that there is $u_0|x_0^{-1}\Wo U$ such that $\psi(x)=\alpha$ (say) for all $x|u_0^{-1}x_0^{-1}\Wo U$ and $\psi(u_0)=\alpha$ or $\alpha\pm(f_2-f_1)$. Thus, as $\psi_1$ is constant on $x_0^{-1}\Wo$, we conclude that $\psi(x)=\alpha$ for all $x|x_0^{-1}\Wo$. If $u_0|U$ with $\psi(u_0)=\alpha+f_2-f_1$, then swapping $u_0|U$ for $x|x_0^{-1}\Wo$ results in a new product decomposition $W'$ such that $\sig(W')=\sig(W)$, $\sigma(W'_0)=f_2$, and $\psi_2$ is not constant on $x_0^{-1}{W_0'}^{(2)}$. However repeating the argument from the beginning of the paragraph for $W'$, using Lemma \ref{lem-Pertebation-II}.5 in place of Lemma \ref{lem-Pertebation-II}.4, we see that $\psi_2$ must be constant on $x_0^{-1}{W_0'}^{(2)}$, a contradiction. Thus we see that any type III swap between $u|U\in\Dc_1\cap \C_2$ and $x|x_0^{-1}\Wo$ results in a product decomposition $W'$ with $\sigma(W'_0)=\sigma(W_0)$. As a result, since $Z\in \Dc_2$ with $\sigma(Z)=f_1+f_2$, it follows from Lemma \ref{lem-Pertebation-III}.1 that (a) holds, contrary to assumption. So we may assume $\Dc_2\cap \C_0$ is empty. Thus, in view of $\Dc_1\cap \C_2\neq \emptyset$ and the hypotheses, it follows that there is $U\in \Dc_2\cap \C_1$.

Performing type II swaps between some $y|U$ and each $x|x_0^{-1}\Wo$ (using the same $R|W_0^{(1)}$ for every swap), we conclude from Lemma \ref{lem-Pertebation-II}.1 that $\psi_1$ is constant on $x_0^{-1}\Wo$. Consequently,
performing type III swaps between $W_0$ and each $V_i\in \Dc_1\cap \C_2$, we conclude from Lemmas \ref{lem-Pertebation-II}.3 and \ref{lem-basic-exchange}.3 that there exists $v_i|V_i$ such that $\psi(x)=\alpha$ (say) for all $x|v_i^{-1}x_0^{-1}\Wo V_i$; moreover, $\psi(v_i)=\alpha$ or $\alpha+f_2-f_1$. If there is $Z\in \Dc_1\cap \C_0$, then, performing type III swaps between the $x|x_0^{-1}\Wo$ and $z|Z$, and between the $x|V_i\in \Dc_1\cap \C_2$ and $z|Z$, we conclude from Lemmas \ref{lem-Pertebation-III}.4 and \ref{lem-Pertebation-III}.5 that $\psi(x)=\alpha$ for all $x|Z$.

If $Z\in \Dc_1\cap \C_0$ does not exist, then $|\Dc_1|\geq 2$ and $|\Dc_2\cap\C_1|\geq 1$ ensure $|\Dc_1\cap \C_2|\geq 2$, and, performing type III swaps between the $V\in \Dc_1\cap \C_2$, we conclude from Lemma \ref{lem-Pertebation-III}.2 that $\psi(x)=\alpha$ for all $x|V$ with $V\in \Dc_1\cap \C_2$, completing the proof. On the other hand, if there is $Z\in \Dc_1\cap \C_0$, then applying type III swaps between $Z$ and each $V_i\in \Dc_1\cap \C_2$, we conclude from Lemma \ref{lem-Pertebation-II}.5 that $\psi_2$ is constant on $V_i$ and $Z$; consequently, since $\psi(v_i)=\alpha$ or $\alpha+f_2-f_1$, and since $\psi(v)=\alpha$ for all $v|v_i^{-1}V_i$, we conclude that $\psi(v_i)=\alpha$ as well, completing the proof.

\smallskip

\noindent CASE 3: \  $W_0\notin C_0$ \ and \ $W\in \Omega_0^{u}$.

Then $W_0\in \C_1$ and $\Dc_1\subset \C_0$ (else (a) holds in view
of Lemma \ref{lem-Pertebation-I}.1). Since $|\Dc_{2}|\geq 1$, there
is $U\in \A_2^*\cap \C_1$. Performing type II swaps between each $x|x_0^{-1}W_0$ and some fixed $u|U$ (using the same fixed sequence $R|W_0^{(1)}$ in each swap), it follows from Lemma \ref{lem-Pertebation-I}.1 that
$\psi(x)=\alpha$ (say) for all $x|x_0^{-1}\Wo$. Let $V_i\in \Dc_1$.
Performing type III swaps between $W_0$ and $V_i$, we conclude from
Lemmas \ref{lem-Pertebation-I}.2 and \ref{lem-basic-exchange}.3 that
$\psi(z)=\alpha$ for all $z|v_i^{-1}V_i$, for some $v_i|V_i$;
moreover, either $\psi(v_i)=\alpha$ or
$\psi(v_i)=\alpha-\sigma(W_0)+\sigma(V_i)$. However, in the latter
case, since $V_i\in \C_0$ and $W_0\in \C_1$ (so that $\sigma(W_0)=f_1$ and $\sigma(V_i)=Cf_1+f_2$, for some $C\in \Z$), we see that
$\psi_2(v_i)\neq \psi_2(\alpha)$. Since $|\Dc_{1}|\geq 2$,
performing type III swaps between the $V_i\in \Dc_1$, we conclude
from Lemma \ref{lem-Pertebation-I}.3 that $\psi_2$ is constant on
each $V_i$, whence $\psi_2(v_i)\neq \psi_2(\alpha)$ is impossible.
Thus $\psi(z)=\alpha$ for all $z|V_i$ with $V_i\in \Dc_1$,
completing the proof.
\end{proof}

\medskip

Lemma \ref{lem-newstuff} allows us to conclude detailed information concerning the values of $\psi$ on $W_0^{(1)}$. Depending on $\sigma(W_j)$ and $\sigma(W_0)$, the appropriate part of Lemma \ref{lem-Pertebation-I}, \ref{lem-Pertebation-II} or \ref{lem-Pertebation-III} will ensure that one of the hypotheses in 1, 2, or 3 holds.

\begin{lemma}\label{lem-newstuff}
Let $W\in \Omega_0$ and  $W_j\in \A_2^*$ be such
that there are $Y|W_j$ and $X|\Wo$ with $|X|=|Y|$ and
$\epsilon'(X,Y)\notin \{1,n\}$, and set
\[
\Dc = \{W'\in \Omega'\mid W'\mbox{ is the result of performing a
type II swap between } X|W_0\mbox{ and } Y|W_j\} \,.
\]
\begin{enumerate}
\item If $\sigma(W'_j)-\sigma(W_j)=0$, for all $W'\in \Dc$, then $|\supp(\psi(W_0^{(1)}))|=1$.

\item If $\sigma(W'_j)-\sigma(W_j)\in \langle f_i \rangle$, where $i\in \{1,2\}$, for all $W'\in \Dc$, then $|\supp(\psi_{3-i}(W_0^{(1)}))|=1$.

\item If $\sigma(W'_j)-\sigma(W_j)\in \{0, F\}$, for all $W'\in \Dc$, where $F\in \Ker(\varphi)$, then $\supp(\psi(W_0^{(1)}))=\{\gamma,\beta\}$ for some $\gamma,\,\beta\in \Ker(\varphi)$ with $\gamma-\beta\in \{0,\pm F\}$.
\end{enumerate}
\end{lemma}

\begin{proof}
1. By hypothesis, there is only one possibility for $\sigma(\psi(R))$, where $R|W_0^{(1)}$ is any subsequence with $|R|=n-\epsilon'(X,Y)$. Furthermore, we have $1\leq |R|\leq n-2<|\psi(W_0^{(1)})|$, and thus 1 follows from Lemma \ref{lem-basic-exchange-ii}.3 applied to $\psi(W_0^{(1)})$.

\smallskip
2. The argument is analogous to that of item 1, using the group
$\Ker (\varphi) /\langle f_i\rangle\cong \langle f_{3-i}\rangle$ in
place of $\Ker(\varphi)$.

\smallskip
3. By the arguments for item 1, replacing Lemma \ref{lem-basic-exchange-ii}.3 by Lemma \ref{lem-basic-exchange-ii}.1, we conclude that  $\psi(W_0^{(1)})=\gamma^l\beta^{n-1-l}$ (say), where $l\geq n-1-l\geq 1$ and $\gamma\neq \beta$ (else the lemma is complete); moreover,
$$\epsilon(X,Y)ne_1
+\sigma(\psi(X))-\sigma(\psi(Y))+\min\{t,\,l\}\cdot\gamma
+(t-\min\{t,\,l\})\cdot\beta+\{0,\beta-\gamma\} =\{0,\,F\},$$ where
$t=n-\epsilon'(X,Y)$. Thus $\beta-\gamma=\pm F$, as desired.
\end{proof}

\medskip

The following lemma encapsulates an alignment argument for the
$\iota$ values that forces them to live in near disjoint intervals.
It will be a key part of the more difficult portions of CLAIM C.

\begin{lemma}\label{lem-aligment}
Let $W\in \Omega_0$, let $\Dc\subset \A^*_2$ be nonempty, and let
$Z|\Wo$ be nontrivial. For $x|S$, let $\psi_0(x)=\psi(x)$, and for
$x\in \Ker(\varphi)$, let $\psi_0$ be the identity map. Let $i\in
\{0,1,2\}$. If $\psi_i(ne_1)\neq 0$ and \be\label{alignment-hyp}
\psi_i(x)-\psi_i(y)+\psi_i(\epsilon(x,y)ne_1)=0\ee for every $x|Z$
and $y|U\in \Dc$, then there exist intervals $J_1$, $J_2$ and $J_3$
of $\Z$ with either \ber\label{vougue1}
\supp(\iota(\prod_{U\in \Dc}U))\subset J_3,\; \supp(\iota(Z))\subset
J_1\cup J_2,&\mbox{ and }& \max J_1\leq \min J_3\leq \max J_3< \min
J_2,\,\mbox{ or}\\ \label{vogue2} \supp(\iota(Z))\subset
J_3,\;\supp(\iota(\prod_{U\in \Dc}U))\subset J_1\cup J_2,&\mbox{ and
}& \max J_1<\min J_3\leq \max J_3\leq \min J_2.\eer Moreover, $I$
can be chosen such that{\rm \,:} \begin{enumerate}
                   \item $\min I$ is congruent to an element in $\iota(Z)$ modulo $n$,
                   \item $\iota(x)\leq \iota(y)$ and $\epsilon(x,y)=0$ for all $x|Z$ and $y|U\in \Dc$, and
                   \item $\psi_i(x)=\psi_i(y)$ for all $xy|Z\prod_{U\in \Dc}U$.
                 \end{enumerate}
\end{lemma}

\begin{proof}

Observe, for $xy|S_2$, that \be\label{align-eps-choice} \epsilon(x,y)=\left\{
                                            \begin{array}{ll}
                                              0, & \iota(x)\leq \iota(y); \\
                                              1, & \iota(x)> \iota(y).
                                            \end{array}
                                          \right.\ee
Consequently, we conclude from (\ref{alignment-hyp}) that \be\label{gal1}\psi_i(x)=\psi_i(y),\ee for all $x|Z$ and
$y|U\in \Dc$ with $\iota(x)\leq \iota(y)$, and that
\be\label{gal2}\psi_i(x)=\psi_i(y)-\psi_i(ne_1),\ee for all $x|Z$ and $y|U\in\Dc$
with $\iota(x)>\iota(y)$.

If there do not exist $x|Z$ and $yy'|\prod_{U\in \Dc}U$ with
$\iota(x)\leq \iota(y)$ and $\iota(x)>\iota(y')$, then, for every
$x|Z$, we have either $\iota(x)\leq \iota(y)$ for all $y|\prod_{U\in
\Dc}U$, or $\iota(x)>\iota(y)$ for all $y|\prod_{U\in \Dc}U$. Thus
we see that (\ref{vougue1}) holds (with $J_3 = [\min
(\supp(\iota(\prod_{U\in \Dc}U))), \max (\supp(\iota(\prod_{U\in
\Dc}U)))]$, $J_1$ being any nonempty interval containing those
$\iota(x)$ with $\iota(x)\leq \iota(y)$ for all $y|\prod_{U\in
\Dc}U$ and $\max J_1 \le \min J_3$, and $J_2$ being any nonempty
interval containing those $\iota(x)$ with $\iota(x)> \iota(y)$ for
all $y|\prod_{U\in \Dc}U$ and $\min J_2 > \max J_3$).

Now instead let $x|Z$ and $yy'|\prod_{U\in \Dc}U$ with $\iota(x)\leq
\iota(y)$ and $\iota(x)>\iota(y')$, and factor $\prod_{U\in
\Dc}U=J'_1J'_2$, where $J'_1$ are those terms $a|\prod_{U\in \Dc}U$
with $\iota(a)<\iota(x)$, and $J'_2$ are those terms $b|\prod_{U\in
\Dc}U$ with $\iota(b)\geq \iota(x)$. By assumption, both $J'_i$ are
nontrivial. Moreover, from (\ref{gal1}) and (\ref{gal2}) and
$\psi_i(ne_1)\neq 0$, we see that
\be\label{guy1}\psi_i(b)=\psi_i(x)\ee and
\be\label{guy2}\psi_i(a)=\psi_i(x)+\psi_i(ne_1)\neq \psi_i(x),\ee
for all $a|J'_1$ and $b|J'_2$. Thus $\psi_i$ is constant on $J'_1$
and also on $J'_2$ but the two values assumed are distinct. If there
were $x'|Z$ such that $\iota(x')\leq \max (\supp(\iota(J'_1)))$,
then by (\ref{gal1}) and (\ref{guy1}) we would conclude that
$\psi_i(x')=\psi_i(b)=\psi_i(x)$, where $b$ is any term of $J'_2$,
while by applying (\ref{gal1}) and (\ref{guy2}) between $x'$ and
$\max(\supp(\iota(J'_1))):=a_0$, we would conclude that
$\psi_i(x')=\psi_i(a_0)=\psi_i(x)+\psi_i(ne_1)\neq \psi_i(x)$, a
contradiction to what we have just seen. We likewise obtain a
contradiction if there were $x'|Z$ such that $\iota(x')>
\min(\supp(\iota(J'_2)))$. Therefore we see that (\ref{vogue2})
holds with $ J_1 = [\min(\supp(\iota(J_1'))), \max(\supp(\iota(J_1')))]$, $J_2 = [\min(\supp(\iota(J_2'))),\max(\supp(\iota(J_2')))]$, and $J_3=[\min(\supp(\iota(Z))),\max(\supp(\iota(Z)))]$.

Choosing $I$ such that $\min I$ is congruent  to $\min
(\supp(\iota(Z)))$ modulo $n$, if either (\ref{vogue2}) holds or else (\ref{vougue1}) holds with
$\supp(\iota(Z))\cap J_2 = \emptyset$, and congruent to $\min
(\supp(\iota(Z))\cap J_2)$ otherwise, the remaining properties
follow in view of (\ref{alignment-hyp}) and
(\ref{align-eps-choice}).
\end{proof}

\medskip
Now we choose a product decomposition $W \in \Omega_0 $, and if
$\Omega^u_0 \neq \emptyset$, we assume that $W\in \Omega^u_0$.

\medskip
\centerline{{\bf CLAIM A:} \ $\mathsf h(S_1)\geq |S_1|-1$.}
\medskip

\begin{proof}
We need to show that there exists $x_0|S_1$ such that
$\psi(x)=\psi(y)$ for all $xy|x_0^{-1}S_1$. We divide the proof into
four main cases. In many of the cases, we obtain partial works towards showing $\mathsf h(S_1)= |S_1|$, which will later be utilized in CLAIM B.

\smallskip

\noindent CASE 1: \ $\Omega_0^u\neq \emptyset$, \ $|\A_1|\geq 2$ \
and \ $|\C_1\cap \A_1|\geq 1$.

In this case, we will moreover show that $\mathsf h(S_1)=|S_1|$
unless $|\A_1\cap \C_0|=1$ or $|\A_1\cap \C_1|=1$, and that
$|\supp(\psi(U))|>1$ for $U\in \A_1\cap \C_i$, where $i\in
\{1,\,2\}$, is only possible when $|\A_1\cap \C_i|=1$.

If $U,\,V\in \A_1$ are distinct, then we can perform a type I swap
between $U$ and $V$, and  by (\ref{exchange-setup-I}) and Lemma
\ref{lem-Pertebation-I}, we conclude that \be\label{pos-swaps-I}
                   \begin{array}{ll}
                     \sigma(\psi(X))-\sigma(\psi(Y))=0, & \hbox{if } U,\,V\in \C_1\\
                     \sigma(\psi(X))-\sigma(\psi(Y))\in \{0,\,(1-C)f_1-f_2\}, & \hbox{if } U\in \C_1,\,V\in \C_0 \hbox{ and } \sigma(V)=Cf_1+f_2\\
                     \sigma(\psi(X))-\sigma(\psi(Y))\in \langle f_1 \rangle, & \hbox{if } U,\,V\in \C_0,
                   \end{array}\ee
for $X|U$ and $Y|V$ with $|X|=|Y|$.

If $|\A_1\cap \C_0|\geq 2$, then using (\ref{pos-swaps-I}) (running
over all $X$ and $Y$ with $|X|=|Y|=1$), we conclude that
$\psi(x)-\psi(y)\in \langle f_1 \rangle$ for all $x$ and $y$
dividing a block from $\A_1\cap \C_0$.

If $|\A_1\cap \C_1|\geq 2$, then using (\ref{pos-swaps-I}) (running over all $X$ and $Y$ with $|X|=|Y|=1$) and Lemma \ref{lem-basic-exchange}.1, we conclude that $\psi(x)=\psi(y)$ for all $x$ and $y$ dividing a block from $\A_1\cap \C_1$.

If $U\in \A_1\cap \C_1$ and $V\in \A_1\cap \C_0$ with $U$ and $V$
distinct, then, using (\ref{pos-swaps-I}) (running over all $X$ and
$Y$ with $|X|=|Y|\leq 2\leq n-1$) and Lemma
\ref{lem-basic-exchange}.3, we conclude that $\psi(x)=\alpha$ (say) for all $x|x_0^{-1}UV$, for some $x_0|UV$; moreover, $\psi(x_0)=\alpha$ or $\alpha\pm ((1-C)f_1-f_2)$.

Suppose $x_0|U$ and $\psi(x_0)\neq \alpha$. Then in view of the
fourth paragraph of CASE 1, we see that $|\A_1\cap \C_1|=1$. Thus
performing type I swaps between $U$ and all possible $V\in \A_1\cap
\C_0$ completes CLAIM A, for $n\geq 5$ or $U\neq W_0$, and, when
$n=3$ and $U=W_0$, we instead conclude that either
$\psi(V)=\alpha^n$ or $\psi(V)=\beta^n$, where
$\psi(W_0^{(1)})=\alpha\beta$, for all $V\in \A_1\cap \C_0$.
However, if there are $V,\,V'\in \A_1\cap \C_0$ with
$\psi(V)=\alpha^n$ and $\psi(V')=\beta^n$ and $\alpha\neq \beta$, then (\ref{pos-swaps-I})
implies that $\beta-\alpha=(1-C)f_1-f_2$ and
$\alpha-\beta=(1-C')f_1-f_2$, where $\sigma(V)=Cf_1+f_2$ and
$\sigma(V')=C'f_1+f_2$, from which we conclude that
$(2-C'-C)f_1-2f_2=0$, contradicting that $m\geq 3$. So we may
instead assume $x_0|V$.

In this case, in combination with the results of the previous
paragraphs, we find that there is at most one $v_i|V_i$, for each
$V_i\in \A_1\cap C_0$, such that $\psi(x)=\alpha$ for all $x|S_1$
not equal to any $v_i$. In this scenario, CLAIM A is done unless we
have two distinct $V_1,\,V_2\in \A_1\cap \C_0$ such that
$\psi(v_1)\neq \alpha$ and $\psi(x)=\alpha$ for all
$x|v_1^{-1}v_2^{-1}UV_1V_2$. However, applying a type I swap
between $y|U$ and $v_1|V_1$, we conclude from (\ref{pos-swaps-I})
that $\alpha-\psi(v_1)= (1-C)f_1-f_2\notin \langle f_1 \rangle$, for
some $C\in \Z$, which, in view of $\alpha\psi(v_1)|\psi(V_1)$,
contradicts the conclusion of the third paragraph of CASE 1. This
completes CASE 1.

\smallskip
\noindent
CASE 2: \ $|\A_1|=1$.

In this case, we will show that $\mathsf h(S_1)=|S_1|$.

Suppose $W_0\in \C_0$. Then we may choose $f_2$ such that
$\sigma(W_0)=f_1+f_2$, and if $\Omega_0^u=\emptyset$, such that
$\sig(W)=f_1^{m-1}f_2^{m-1}(f_1+f_2)$ also. Let $\Dc_1$ be those
blocks $W_i$ with $\sigma(W_i)=f_1$ and let $\Dc_2$ be all other
blocks from $\A_2^*$. Applying Lemma \ref{lemma-for-two-opts}, we
see that Lemma \ref{lemma-for-two-opts}(ii) must hold, else
$ge_1+e_2+\alpha$ will have multiplicity at least $mn-1$ in $S$, as
desired. Performing a type II swap between the $X|W_0$ and $Y|W_j$
given by Lemma \ref{lemma-for-two-opts}(ii), we conclude, from
Lemmas \ref{lem-newstuff}.2 and either \ref{lem-Pertebation-I}.2
(since $\sigma(W_0)=f_1+f_2$) or \ref{lem-Pertebation-II}.4, that
$\psi_1$ is constant on $W_0^{(1)}$. However, reversing the roles of
$\Dc_1$ and $\Dc_2$ and repeating the above argument using Lemmas
\ref{lem-Pertebation-I}.3 and \ref{lem-Pertebation-II}.5 in place of
Lemmas \ref{lem-Pertebation-I}.2 and \ref{lem-Pertebation-II}.4, we
conclude that $\psi_2$ is also constant on $W_0^{(1)}$, whence
$\psi$ is constant on $W_0^{(1)}$, completing the proof of CLAIM A.
So we may assume $W_0\notin \C_0$.

Suppose $\Omega_0^u=\emptyset$. Then we may w.l.o.g. assume
$\sig(W)=f_1^{m-1}f_2^{m-1}(f_1+f_2)$, that $\C_1$ consists of those
blocks $W_i$ with $\sigma(W_i)=f_1$, and that $\sigma(W_0)=f_1$. Let
$\Dc_1=\C_2$ and $\Dc_2=\C^*_1\cup \C_0$. Applying Lemma
\ref{lemma-for-two-opts}, we see that Lemma
\ref{lemma-for-two-opts}(ii) must hold, else there will be a term
with multiplicity at least $mn-1$ in $S$, as desired. Thus Lemmas
\ref{lem-newstuff}.3 and \ref{lem-Pertebation-II}.3 imply that
$\supp(\psi(W_0^{(1)}))=\{\gamma,\,\beta\}$ (say) with
$\beta-\gamma=\pm (f_2-f_1)$ (else CLAIM A follows).

Reversing the roles of $\Dc_1$ and $\Dc_2$ and again applying Lemma \ref{lemma-for-two-opts}, we once more see that Lemma \ref{lemma-for-two-opts}(ii) must hold, else there is a term with multiplicity $mn-1$ in $S$, as desired. Thus Lemma \ref{lem-newstuff}.2 and either Lemma  \ref{lem-Pertebation-II}.1 or \ref{lem-Pertebation-II}.4  imply that $\psi_1$ is constant on $W_0^{(1)}$, contradicting that $\beta-\gamma=\pm (f_2-f_1)$. So we may assume $\Omega_0^u\neq \emptyset$.

Let w.l.o.g. $W_1,\ldots, W_{m-2}$ be the blocks of $\C^*_1\cap
\A_2$, and let $\Dc_1=\C^*_1$ and $\Dc_2=\C_0$.  Apply Lemma
\ref{lemma-for-two-opts}. If Lemma \ref{lemma-for-two-opts}(ii)
holds, then Lemmas \ref{lem-newstuff}.1 and
\ref{lem-Pertebation-I}.1 imply that $\psi$ is constant on
$W_0^{(1)}$, whence  CLAIM A follows. Therefore we may instead
assume $\iota(x)=g$ and $\psi(x)=\alpha$ (say) for all terms
$x|x_0^{-1}\Wo W_1\ldots W_{m-2}$, for some $x_0|\Wo$ with
$\iota(x_0)\equiv g+1\mod n$.

Consider $W_j$ with $j\geq m-1$. If $\iota(W_j)\neq g^n$, then there
exist $x|\Wo$ and $y|W_j$ with $\epsilon'(x,y)\notin \{1,n\}$,
whence Lemmas \ref{lem-newstuff}.3 and \ref{lem-Pertebation-I}.2
imply that $\supp(\psi(W_0^{(1)}))=\{\gamma,\,\beta\}$ (say) with
$\beta-\gamma=\pm F_j$ (else  CLAIM  A follows), where
$F_j=(1-C_j)f_1-f_2$ and $\sigma(W_j)=C_jf_1+f_2$.

If $W_{k}$ is another block with $k\geq m-1$ and $\iota(W_{k})\neq
g^n$, then the above paragraph implies that $\beta-\gamma=\pm
F_{k}$, where $F_{k}=(1-C_{k})f_1-f_2$ and
$\sigma(W_k)=C_{k}f_1+f_2$. Thus, since $m\geq 3$ and
$\beta-\gamma=\pm F_j$, we conclude that $F_j=F_k$ and $C_j\equiv
C_{k}\mod m$. As a result, we see that any two blocks $W_j$ and
$W_{k}$, with $j,\,k\geq m-1$ and  $\iota(W_{k}),\,\iota(W_{j})\neq
g^n$, must have $\sigma(W_j)=\sigma(W_{k})$. Hence, since $W\in
\Omega^u_0$, we conclude that there are at least two distinct blocks
$W_s$ and $W_{r}$ with $s,\,r\geq m-1$ and
$\iota(W_{s})=\iota(W_{r})=g^n$. Performing type III swaps between
$W_0$ and both $W_s$ and $W_{r}$, we conclude from Lemmas
\ref{lem-Pertebation-I}.2 and \ref{lem-basic-exchange}.3 that
$\psi(x)=\alpha$ for all but at most two terms of $W_sW_{r}$, whence
$ge_1+e_2+\alpha$ has multiplicity at least $(m-1)n-1+2n-2\geq mn$
in $S$, contradicting that $S\in \A(G)$ and completing CASE 2.

\smallskip

\noindent CASE 3: \ $\Omega_0^u\neq \emptyset$, \ $|\A_1|\geq 2$ \
and \ $|\C_1\cap \A_1|=0$.

In this case, we will moreover show that $\mathsf h (S_1)=|S_1|$.

We may w.l.o.g. assume $W_1,\ldots,W_{m-1}$ are the blocks in
$\C_1\cap \A_2$. Let $\Dc_1=\C_1$ and $\Dc_2=\C_0^*\cap \A_2$. If
$|\Dc_2|\geq 1$, then we can apply Lemma \ref{lemma-for-two-opts}.
Otherwise, in view of Lemma \ref{lem-Pertebation-I}.2, we may assume
hypothesis (a) holds in Lemma \ref{lemma-for-two-opts}, else
applying CASE 1 to the resulting product decomposition $W'$  would
imply, in view of $|\Dc_2|=0$, that $\psi(x)=\alpha$ (say) for all
$x|W'_i=W_i$ with $i\in [m,2m-2]$, in which case
$\sigma(W'_i)=ne_1+n\alpha$ has multiplicity $m-1$ in $\sig(W')$,
contradicting that $\sig(W')=\sig(W)$ (in view of Lemma
\ref{lem-Pertebation-I}.2) with $W\in \Omega_0^u$. Thus, in either
case Lemma \ref{lemma-for-two-opts} is available. If Lemma
\ref{lemma-for-two-opts}(i) holds, then $ge_1+e_2+\alpha$ is a term
with multiplicity at least $mn-1$ in $S$, as desired. Therefore
there is $X|\Wo$ and $Y|W_j$, for some $j\in [1,m-1]$, such that
$|X|=|Y|$ and $\epsilon'(X,Y)\notin \{1,n\}$. Hence Lemmas
\ref{lem-newstuff}.3 and \ref{lem-Pertebation-I}.2 imply that
$\supp(\psi(W_0^{(1)}))=\{\gamma,\,\beta\}$ (say) with
$\gamma-\beta\in \{0,\pm F\},$ where $F=(C-1)f_1+f_2$ and
$\sigma(W_0)=Cf_1+f_2$. Since $|\A_1|\geq 2$, let $V\in \C^*_0\cap
\A_1$. Performing type I swaps between $W_0$ and $V$, we conclude
from Lemma \ref{lem-Pertebation-I}.3 that $\psi_2$ is constant on
$VW_0^{(1)}$, whence $\gamma-\beta\in \{0,\pm F\}$ implies
$\gamma=\beta$.

Performing type I swaps among the $V\in \C_0\cap \A_1$, we conclude
from Lemma \ref{lem-Pertebation-I}.3 that $\psi_2(x)=\psi_2(\gamma)$
for all $x|V\in \C_0\cap \A_1$. Let
$W'$ be the product decomposition resulting from performing a type
II swap between $X|W_0$ and $Y|W_j$ (with $X$ and $Y$ as given by Lemma \ref{lemma-for-two-opts}(ii) in the previous paragraph). Since $\epsilon'(X,Y)\notin
\{1,n\}$, we conclude that there is a block $W'_k\in \C_1$, with
$k\in \{0,j\}$, having $(e_1+\gamma)|W'_k$. Since $\sig(W')=\sig(W)$ (in view
of Lemma \ref{lem-Pertebation-I}.2), performing type I swaps between
$W'_k$ and each distinct block $V'=V\in \C^*_0\cap \A_1$, we conclude
from Lemma \ref{lem-Pertebation-I}.2 that either $\psi(x)=\gamma$ or
$\psi(x)=\gamma+\sigma(V')-\sigma(W'_k)$, for each $x|V'$. However,
since $W'_k\in \C_1$ and $V'\in \C_0$, it follows that the latter
contradicts that $\psi_2$ is constant on $V|W_0^{(1)}$ with value $\psi_2(\gamma)$. Therefore we
conclude that $\psi(x)=\gamma$ for all $x|V'$, with $V'=V^*\in \C_0\cap
\A_1$, whence $\psi(x)=\gamma$ for all $x|S_1$, as desired,
completing CASE 3.

\smallskip

\noindent CASE 4: \ $\Omega_0^u= \emptyset$ \ and \ $|\A_1|\geq 2$.

We may w.l.o.g. assume $\widetilde{\sigma} (W) =
f_1^{m-1}f_2^{m-1}(f_1+f_2)$, by an appropriate choice of $f_2$,
whence  CLAIM A follows easily by performing type I swaps between
the blocks of $\A_1$ and using Lemmas \ref{lem-Pertebation-III} and
\ref{lem-basic-exchange}. This completes CASE 4.
\end{proof}

In view of CLAIM A, we may assume $S_1=e_1^{|S_1|-1}(e_1+a)$, for
some $a\in \Ker(\varphi)$. Let $y_0=e_1+a$.

\medskip
\centerline{{\bf CLAIM B:} \ $\mathsf h (S_1)=|S_1|$.}

\begin{proof}
We assume by contradiction $a\neq 0$. In view of the partial
conclusions of CLAIM A, we may assume $|\A_1|\geq 2$ (in view of
CASE 2 of CLAIM A), and, if $\Omega_0^u\neq \emptyset$, that
$|\A_1\cap \C_1|\geq 1$ (in view of CASE 3 of CLAIM A). We proceed
in four cases.

\smallskip

\noindent CASE 1: \ $\Omega_0^u\neq \emptyset$ \ and $y_0|U$ for some $U\in \A_1\cap \C_1$.

In view of CASE 1 of CLAIM A, we have $|\A_1\cap \C_1|=1$. Hence, if
$U\neq W_0$, then $W_0\in\C_0$, and performing a type I swap between
$y_0|U$ and some $y|W_0$ results (in view of Lemma
\ref{lem-Pertebation-I}.2) in a new product decomposition $W'$ with
$\sig(W')=\sig(W)$, $U'\in \C_0$, $W_0'\in \C_1$, $y_0|W'_0$ and
$W'$ also satisfying the hypothesis of CASE 1. On the other hand,
if $U=W_0$, then $|\A_1|\geq 2$ and $|\A_1\cap \C_1|=1$ imply that
there is $V\in \A_1^*\cap \C_0$, and performing a type I swap
between $y_0|W_0$ and some $y|V$ results (in view of Lemma
\ref{lem-Pertebation-I}.2) in a new product decomposition $W'$ with
$\sig(W')=\sig(W)$, $W_0'\in \C_0$, $V'\in \C_1$, $y_0|V'$ and $W'$
also satisfying the hypothesis of CASE 1. Thus w.l.o.g. we may
assume $U\neq W_0$. Since $U\in \C_1$ and $\sig(W')=\sig(W)$ with $W'_0\in \C_1$
(with $W'$ as in the second sentence of CASE 1), then, letting
$\sigma(W_0)=Cf_1+f_2$, we see that $a=(1-C)f_1-f_2$.

Let $\Dc_1=\A_2^*(W')\cap \C_1(W')$ and $\Dc_2=\A_2^*(W')\cap
\C_0(W')$. Since $|\A_1\cap \C_1|=1$ and $W'_0\in \C_1$, we have $|\Dc_1|=m-2$, and by CLAIM
A we have $|\Dc_2|\geq 1$ (else $e_1$ is a term with multiplicity at
least $(m+1)n-2\geq mn$, contradicting that $S\in \A(G)$). If Lemma
\ref{lemma-for-two-opts}(ii) holds for $W'$, then Lemmas
\ref{lem-newstuff}.1 and \ref{lem-Pertebation-I}.1 imply that $a=0$,
a contradiction. Therefore Lemma \ref{lemma-for-two-opts}(i) holds
for $W'$. Let $g$ and $\alpha$ be as given by Lemma \ref{lemma-for-two-opts}(i).

Since $|\Dc_2|\geq 1$, let $V\in \A_2^*(W)\cap \C_0(W)$. If
$\iota(V)=g^n$, then, performing type III swaps between $V$ and some
$Z\in\A_2^*\cap \C_1$, and between $V$ and $W_0$, we conclude from
Lemmas \ref{lem-Pertebation-I}.2, \ref{lem-Pertebation-I}.3 and \ref{lem-basic-exchange}.3 that
$\psi(x)=\alpha$ for all $x|V$, whence $ge_1+e_2+\alpha$ has
multiplicity at least $mn-1$ in $S$, as desired. Therefore, in view
of $\iota(\Wo)\equiv g^{n-1}(g+1)\mod n$, we see that there exists
$x|W_0^{(2)}={W_0'}^{(2)}$ and $y|V=V'$ such that
$\epsilon'(x,y)\notin \{1,n\}$. Hence, from Lemmas
\ref{lem-newstuff}.3 (applied to $W'$) and
\ref{lem-Pertebation-I}.2, it follows that $a=\pm((1-C')f_1-f_2)$,
where $\sigma(V)=C'f_1+f_2$. Thus, since $a=(1-C)f_1-f_2$ and $m\geq
3$, we conclude that $C'f_1=Cf_1$ and $\sigma(V)=\sigma(W_0)$. As $V\in
\A_2^*(W)\cap \C_0(W)$ was arbitrary, we see that
$\sigma(V)=Cf_1+f_2$ for all $V\in \A_2(W)\cap \C_0(W)$. On the
other hand, if $Z\in \A_1(W)\cap \C_0(W)$, then, performing type I
swaps between $U$ and $Z$, we conclude from Lemma
\ref{lem-Pertebation-I}.2 that $a=(1-C'')f_1-f_2$, where
$\sigma(Z)=C''f_1+f_2$. Thus $a=(1-C)f_1-f_2$ implies that $C''f_1=Cf_1$,
and now $\sigma(Z)=Cf_1+f_2$ for all $Z\in \A_1(W)\cap \C_0(W)$.
Consequently, $\sigma(Z)=Cf_1+f_2$ for all $Z\in \C_0(W)$,
contradicting that $\mathsf h (\sig(W))<m$. This completes CASE 1.

\smallskip

\noindent CASE 2: \ $\Omega_0^u\neq \emptyset$ \ and $y_0|U$ for some $U\in \A_1\cap \C_0$

Recall that $|\A_1\cap
\C_1|\geq 1$ and $|\A_1|\geq 2$. CASE 1 of CLAIM A and the
hypothesis of CASE 2 further imply that $|\A_1\cap \C_0|=1$. Thus,
if $U\neq W_0$, then $W_0\in \C_1$, and performing a type I swap
between $y_0|U$ and some $y|W_0$ results (in view of Lemma
\ref{lem-Pertebation-I}.2) in a product decomposition $W'$ with
$y_0|W'_0$, $W'_0\in \C_0$, $\sig(W')=\sig(W)$ and $W'$ satisfying
the hypotheses of CASE 2. Thus w.l.o.g. we may assume $U=W_0$.

Since $|\A_1\cap \C_1|\geq 1$, let $V\in \A_1^*\cap \C_1$.
Performing a type I swap between $y_0|W_0$ and some $y|V$, letting
$W'$ be the resulting product decomposition, we conclude from Lemma
\ref{lem-Pertebation-I}.2 that $a=(C-1)f_1+f_2$, where
$\sigma(W_0)=Cf_1+f_2$. Since $|\A_1\cap \C_0|=1$ and $W_0\in \C_0$,
let w.l.o.g. $W_1,\ldots, W_{m-1}$ be the blocks of $\A^*_2\cap
\C_0$. If $x|\Wo$ and $y|W_j$, with $j\in [1,m-1]$ and
$\iota(x)=\iota(y)$, then, performing a type III swap between
$x|W_0$ and $y|W_j$ and between $x|W'_0$ and $y|W'_j$, we conclude
in view of Lemmas \ref{lem-Pertebation-I}.3 and
\ref{lem-Pertebation-I}.2 that $\psi(x)=\psi(y)$; thus, letting
$\Dc_1=\A_2^*\cap \C_0$ and $\Dc_2=\A_2^*\cap \C_1$, we see that
hypothesis (a) holds in Lemma \ref{lemma-for-two-opts}. If Lemma
\ref{lemma-for-two-opts}(i) holds, then $ge_1+e_2+\alpha$ is a term
of $S$ with multiplicity at least $mn-1$, as desired. Therefore
Lemma \ref{lemma-for-two-opts}(ii) holds, whence Lemmas
\ref{lem-newstuff}.2 and \ref{lem-Pertebation-I}.3 imply that $a\in
\langle f_1 \rangle$, contradicting that $a=(C-1)f_1+f_2$. This
completes CASE 2.

\smallskip

Note that if $\Omega_0^u=\emptyset$, then (in view of $|\A_1|\geq
2$) we may w.l.o.g. assume $y_0|U$ with $U\neq W_0$, by an
appropriate type I swap. Moreover, when $\Omega_0^u=\emptyset$, we
will w.l.o.g. assume $\sig(W)=f_1^{m-1}f_2^{m-1}(f_1+f_2)$ with
$\C_1$ consisting of those blocks $W_i$ with $\sigma(W_i)=f_1$.

\smallskip

\noindent CASE 3: \ $\Omega_0^u=\emptyset$ \ and $y_0|U$ for some  $U\in
\A^*_1\cap \C_0$.

We may w.l.o.g. assume $W_0\in \C_1$. Performing a type I swap
between $y_0|U$ and some $y|W_0$, letting $W'$ be the resulting
product decomposition, we conclude from Lemma
\ref{lem-Pertebation-III}.4 that $a=f_2$. Let $\Dc_1=\A_2^*(W')\cap
\C_2(W')$ and let $\Dc_2=A_2^*(W')\cap \C_1(W')$. Observe that
$|\Dc_1|=m-1$, else performing a type I swap between $y_0|U$ and
some $V\in \A_1\cap \C_2$ would imply in view of Lemma
\ref{lem-Pertebation-III}.5 that $a=f_1$, contradicting that
$a=f_2$. If a type III swap between $W'_0$ and some $W'_j\in \Dc_1$
results in a new product decomposition $W''$ with $\sigma(W''_0)\neq
\sigma(W'_0)$, then Lemma \ref{lem-Pertebation-III}.5 implies
$\sigma(W''_0)=f_2$, whence, performing a type I swap between
$y_0|{W''_0}^{(1)}={W'_0}^{(1)}$ and $U''=U'$, we conclude from Lemma
\ref{lem-Pertebation-II}.3 that $-a=f_1-f_2$, contradicting that
$a=f_2$. Thus hypothesis (a) of Lemma \ref{lemma-for-two-opts} holds
for $W'$. If Lemma \ref{lemma-for-two-opts}(i) holds, then
$ge_1+e_2+\alpha$ has multiplicity at least $mn-1$ in $S$, as
desired. Therefore, Lemma \ref{lemma-for-two-opts}(ii) holds, whence
Lemmas \ref{lem-newstuff}.2 and \ref{lem-Pertebation-II}.5 imply
that $a\in \langle f_1 \rangle$, contradicting that $a=f_2$ and
completing CASE 3.

\smallskip

\noindent CASE 4: \ $\Omega_0^u=\emptyset$ \ and $y_0|U\in
\A^*_1$ \ with $U\notin \C_0$.

We may w.l.o.g. assume $U\in \C_1$. If $W_0\in \C_1$, then performing type I swaps between $W_0$ and $U$ would imply, in view of Lemma \ref{lem-Pertebation-III}.1, that $a=0$, a contradiction. Moreover, this also shows that $\A_1\cap \C_1=\{U\}$.

Suppose  $W_0\in \C_2$. Performing a type I swap between $y_0|U$ and
some $y|W_0$, letting $W'$ be the resulting product decomposition,
we conclude from Lemma \ref{lem-Pertebation-II}.3 that $\sig(W')=\sig(W)$, $W'_0\in
\C_1$, $a=f_1-f_2$ and $ne_1=\sigma(U')=f_2$. Let
$\Dc_1=\A_2^*(W')\cap \C_1(W')$ and let $\Dc_2=\A_2^*(W')\cap
C_0(W')$. Since $\A_1\cap \C_1=\{U\}$, we have $|\Dc_1|=m-2$. Since
$ne_1=f_2\neq f_1+f_2$, we have $Z\in\C_0$ with $Z\in \A_2^*$, and thus
$|\Dc_2|\geq 1$. Apply Lemma \ref{lemma-for-two-opts} to $W'$. If
Lemma \ref{lemma-for-two-opts}(ii) holds, then Lemmas
\ref{lem-newstuff}.1 and \ref{lem-Pertebation-II}.1 imply
$\psi_1(a)=0$, contradicting that $a=f_1-f_2$. Therefore Lemma
\ref{lemma-for-two-opts}(i) holds, whence $gne_1+ne_2+n\alpha=\sigma(V)=f_1$, where $V\in \Dc_1$.
If there is a type III swap between $Z'=Z$ and $W'_0$ resulting in a
product decomposition $W''$ with $\sigma(W''_0)\neq \sigma(W'_0)$,
then Lemma \ref{lem-Pertebation-III}.4 implies that
$\sigma(W''_0)=f_1+f_2$, whence, performing a type I swap between
$y_0|W''_0$ and $y|U''=U'$, we conclude from Lemma
\ref{lem-Pertebation-III}.5 that $-a=-f_1$, contradicting that
$a=f_1-f_2$. Therefore hypothesis (a) holds in Lemma
\ref{lemma-for-two-opts} for $W'$ with the roles of $\Dc_1$ and
$\Dc_2$ reversed. Apply Lemma \ref{lemma-for-two-opts} in this case.
If Lemma \ref{lemma-for-two-opts}(ii) holds, then Lemmas
\ref{lem-newstuff}.2 and \ref{lem-Pertebation-II}.4 imply that $a\in
\langle f_2 \rangle$, contradicting that $a=f_1-f_2$. Therefore
Lemma \ref{lemma-for-two-opts}(i) holds, whence
$gne_1+ne_2+n\alpha=\sigma(Z)=f_1+f_2$, contradicting that
$gne_1+ne_2+n\alpha=f_1$. So we may assume instead that $W_0\in
\C_0$.

Performing a type I swap between $y_0|U$ and some $y|W_0$, letting
$W'$ be the resulting product decomposition, we conclude from Lemma
\ref{lem-Pertebation-III}.4 that $\sig(W')=\sig(W)$, $W'_0\in \C_1$, $a=-f_2$, and
$ne_1=\sigma(U')=f_1+f_2$. Let $\Dc_1=\A_2^*(W')\cap \C_2(W')$. If there
is $V\in \A_1\cap \C_2$, then, performing a type I swap between
$y_0|U$ and some $y|V$, we conclude from Lemma
\ref{lem-Pertebation-II}.3 that $a=f_1-f_2$, contradicting that
$a=-f_2$. Therefore $|\Dc_1|=m-1$. Let $\Dc_2=\A_2^*(W')\cap \C_1(W')$.
Since $\A_1\cap \C_1=\{U\}$, we have $|\Dc_2|\geq m-2$. Thus we may
apply Lemma \ref{lemma-for-two-opts} to $W'$. If Lemma
\ref{lemma-for-two-opts}(i) holds, then $ge_1+e_2+\alpha$ is a term
of $S$ with multiplicity at least $mn-1$, as desired. Therefore
Lemma \ref{lemma-for-two-opts}(ii) holds, whence Lemmas
\ref{lem-newstuff}.3 and \ref{lem-Pertebation-II}.3 imply that
$a=\pm(f_1-f_2)$, contradicting that $a=-f_2$. This completes CASE
4.
\end{proof}

\medskip
There exists $e_2' \in e_2+nG$ such that  $(e_1,e_2')$ is a
basis for $G$.  Thus, after changing notation if necessary, we may
suppose that $(e_1, e_2)$ is a basis of $G$. If $g \in G$ and $x, y
\in \Z$ with $g = xe_1 + ye_2$, then we set $\pi_1 (g) = xe_1$ and
$\pi_2 (g) = y e_2$.

\medskip
\centerline{{\bf CLAIM C:} \ There exists \ $x_0|S_2$ \ such that \
$x-y\in \langle e_1 \rangle$ \ for all \ $xy|x_0^{-1}S_2$.}
\medskip

\begin{proof}
We need to show that there exists $x_0|S_2$ such that
$\pi_2(\psi(x))=\pi_2(\psi(y))$ for all $xy|x_0^{-1}S_2$. We divide
the proof into four cases.

\smallskip

\noindent CASE 1: \ $\Omega_0^u\neq \emptyset$ \ and there is \
$U\in \A^*_1\cap \C_1$.

In this case, we have \be \label{aligment-1}ne_1=\sigma(U)=f_1.\ee
Let $V\in
\A^*_2$. Perform type (II) swaps between $W_0$ and $V$. If
$V,\,W_0\in \C_1$, then we conclude from Lemmas
\ref{lem-Pertebation-I}.1 and \ref{lem-basic-exchange}.1 that
$\pi_2(\psi(x))=\alpha_2$ (say) for all $x|V\Wo$. If $V,\,W_0\in \C_0$,
then we conclude, from Lemmas \ref{lem-Pertebation-I}.3 and
\ref{lem-basic-exchange}.1  and (\ref{aligment-1}), that
$\psi_2$ is constant on $V\Wo$, whence (\ref{aligment-1}) further
implies that $\pi_2(\psi(x))=\alpha_2$ for all $x|V\Wo$. If
$|\{V,\,W_0\}\cap \C_1|=1$, then we conclude from Lemmas
\ref{lem-Pertebation-I}.2 and \ref{lem-basic-exchange}.3 that
$\pi_2(\psi(x))=\alpha_2$ for all $z|x_0^{-1}V\Wo$, for some $x_0|V\Wo$. If
$\pi_2(\psi(x_0))\neq \alpha_2$ and $x_0|V$, then pull $x_0$ up into a new
product decomposition $W'$ and assume we began with $W'$ instead of $W$ (note that (\ref{aligment-1}) holds independent of $W'$ and that $\widetilde{\sigma} (W) = \widetilde{\sigma} (W')$ follows by Lemma
\ref{lem-Pertebation-I}.2, so all previous arguments can be applied to $W'$ regardless of whether $\A^*_1(W')\cap \C_1(W')$ is nonempty or not). Doing this for all $V\in \A^*_2$, we conclude that there
is an $x_0|S_2$ such that $\pi_2(\psi(x))=\alpha_2$ for all $x|x_0^{-1}S_2$, completing CASE 1.

\smallskip

\noindent CASE 2: \ $\Omega_0^u\neq \emptyset$ \ and \ $\A_1\cap
\C_1=\{W_0\}$.

Performing type II swaps between $W_0$ and each $U\in \A^*_2\cap
\C_1$, we conclude from Lemmas \ref{lem-basic-exchange}.1 and
\ref{lem-Pertebation-I}.1 that $\pi_2(\psi(x))=\alpha_2$ (say) for
all $x|\Wo U$, with $U\in \A^*_2\cap \C_1$. Let w.l.o.g. $W_1,\ldots,
W_l$ be the blocks in $\A_2\cap \C_0$, and let
$W_{m+1},\ldots,W_{2m-2}$ be the blocks in $\A^*_2\cap \C_1$. Note
$l\geq 1$ else CLAIM C follows by the previous conclusion.
Performing type II swaps between $W_0$ and $W_j$, with $j\in [1,l]$,
we conclude from Lemmas \ref{lem-basic-exchange}.3 and
\ref{lem-Pertebation-I}.2 that $\pi_2(\psi(x))=\alpha_2$ for all
$x|z_j^{-1}W_j$, for some $z_j|W_j$. We may w.l.o.g. assume
$\pi_2(\psi(z_j))\neq \alpha_2$ for $j\in [1, l']$ and
$\pi_2(\psi(z_j))=\alpha_2$ for $j\in [l'+1,l]$. We have $l'\geq 2$
else CLAIM C follows.

Perform a type II swap between $z_1|W_1$ and any term $y|\Wo$, and let $W'$ denote the resulting product decomposition. Since $\pi_2(\psi(z_1))\neq \alpha_2$, we are assured that $\pi_2(\sigma(W_0))\neq \pi_2(\sigma(W'_0))$, and hence $\sigma(W_0)\neq \sigma(W'_0)$. Thus
Lemma \ref{lem-Pertebation-I}.2 implies that $\sig(W')=\sig(W)$, $W'_0\in \C_0$ and $W'_1\in \C_1$.

Now pull the term $z_2|W_2$ up into a new product decomposition
$W''$. Note by Lemma \ref{lem-Pertebation-I}.2 that
$\widetilde{\sigma} (W'') = \widetilde{\sigma} (W)$. If $W''_0\in
\C_1$, then the arguments of the first paragraph show that
$\pi_2(\psi(z_2))=\alpha_2$, contradicting that $l'\geq 2$. Therefore
$W''_2\in \C_1$ instead. However, noting that $yW_0^{(1)}|W_0''$, for
some $y|\Wo$ (since $\sigma(\iota(\Wo))\equiv 1 \mod n$ and
$\sigma(\iota(W'_2))\equiv 0\mod n$), we can still perform the swap
between $y|W''_0$ and $z_1|W''_1=W_1$ described in the previous paragraph,
which results in a new product
decomposition $W'''$ in which the $m$ blocks
$$W'''_1=W'_1,\,W'''_2=W''_2,\,W'''_{m+1}=W_{m+1},\ldots,W'''_{2m-2}=W_{2m-2}$$
all have equal sum $f_1$, contradicting that $S'\in \A(G)$, and
completing CASE 2.

\smallskip

\noindent CASE 3: \ Either \ ($\Omega_0^u\neq \emptyset$ and
$\A_1\cap \C_1 = \emptyset$) \  or \ ($\Omega_0^u=\emptyset$ and
$W_0\notin \C_0$).

If $\Omega_0^u=\emptyset$, we may w.l.o.g. assume $\sig
(W)=f_1^{m-1}f_2^{m-1}(f_1+f_2)$ with $\C_1$ those blocks with sum
$f_1$ and $\C_2$ those blocks with sum $f_2$, and that $W_0\in
\C_2$. Let w.l.o.g. $W_1,\ldots, W_{s}$ be the $s\leq m-1$ blocks of
$\C_1\cap \A^*_2$. Let $\sigma(W_0)=Cf_1+f_2$ and
$F=(C-1)f_1+f_2$. If $\Omega_0^u\neq \emptyset$, then we have
$s=m-1$ by hypothesis. If $s=0$, then $|\A^*_1\cap \C_1|=m-1$,
implying $e_1$ is a term with multiplicity at least $mn-1$ in $S$
(in view of CLAIM B), as desired. Therefore we may assume $s>0$.

We claim, for any $W$ satisfying the hypothesis of CASE 3 and
notated as above (and in fact, if $W\in \Omega_0^{nu}$, we will not
need that $\Omega_0^u=\emptyset$), that
\be\label{all-but-one-pi2}\pi_2(\psi(x_0^{-1}\Wo\prod_{\nu=1}^{s}W_{\nu}))=q_2^{(s+1)n-1}\ee
for some $x_0|\Wo\prod_{\nu=1}^{s}W_{\nu}$ and $q_2\in \Ker(\varphi)$.
To show this, perform type II swaps between $W_0$ and $W_i$, $i\in
[1,s]$. If $\pi_2(F)=0$, then Lemmas \ref{lem-basic-exchange}.1 and
either \ref{lem-Pertebation-I}.2 or \ref{lem-Pertebation-II}.3 imply
that (\ref{all-but-one-pi2}) holds with $\pi_2(x_0)=q_2$ as well. If
$\pi_2(F)\neq 0$ and (\ref{all-but-one-pi2}) fails, then Lemmas
\ref{lem-basic-exchange}.3 and either \ref{lem-Pertebation-I}.2 or
\ref{lem-Pertebation-II}.3 imply that $\pi_2(\psi(z))=q_2$ (say) for
all $z|x_i^{-1}x_0^{-1}\Wo W_i$, for some  $x_i|W_i$,  $i\in [1,s]$;
moreover, $s\geq 2$ and w.l.o.g. $\pi_2(\psi(x_1))$ and
$\pi_2(\psi(x_2))$ are not equal to $q_2$. Pull $x_1|W_1$ up into a
new product decomposition $W'$. If $\sigma(W'_0)=\sigma(W_0)$, then
the arguments of the previous sentence imply either
$\pi_2(\psi(x_1))=q_2$ or $\pi_2(\psi(x_2))=q_2$, a contradiction.
If $\sigma(W'_0)\neq\sigma(W_0)$ and $W\in \Omega_0^u$, then Lemma
\ref{lem-Pertebation-I}.2 implies that $W'\in \Omega_0^u$ with
$W'_0\in \C_1$, whence CLAIM C follows in view of CASE 2 applied to $W'$. Therefore we may assume
$\sigma(W'_0)\neq\sigma(W_0)$, $W\in \Omega_0^{nu}$ and $W'_0\in
\C_1$ (in view of Lemma \ref{lem-Pertebation-II}.3). Let $y$ be a
term that divides both $\Wo$ and ${W_0'}^{(2)}$ (possible since
$\sigma(\iota(W_0))\equiv 1\mod n$). Choose $I$ such that $\min
I\equiv \iota(y)\mod n$, and consequently $\epsilon(y,z)=0$ for any
$z$ (in view of (\ref{align-eps-choice})). Note that while the new choice of $I$ may change the overall
value of $\psi(x)$, where $x|S_2$, in a nontrivial manner, nonetheless,
the value of $\pi_2(\psi(x))$ remains unchanged. Perform type II
swaps between $y|W_0$ and any $z|W_2$. In view of our choice of $I$,
Lemma \ref{lem-Pertebation-II}.3 and $\pi_2(\psi(x_2))\neq q_2$, we
conclude  that
$-\psi(x_2)+\psi(y)=F=-f_1+f_2$ (since $-\pi_2(\psi(x_2))+\pi_2(\psi(y))\neq 0$, implying $-\psi(x_2)+\psi(y)\neq 0$), and that $-\psi(z)+\psi(y)=0$ if $z\neq x_2$ (since $-\pi_2(\psi(z))+\pi_2(\psi(y))=0$); in particular, $\psi_1(x_2)\neq
\psi_1(z)$ for $z|x_2^{-1}W_2$. However, performing type II swaps
between $y|W'_0$ and any $z|W'_2=W_2$, we conclude from Lemma
\ref{lem-Pertebation-II}.1 and the choice of $I$ that $\psi_1$ is constant on $W'_2=W_2$,
contradicting the previous sentence. Thus (\ref{all-but-one-pi2}) is
established in all cases.

Next we proceed to show that  $s=m-1$. To this end, suppose $s<m-1$.
As noted before, we may then assume $\Omega_0^u=\emptyset$.  Let
$U\in \A^*_1\cap \C_1$ (which is nonempty by the assumption
$s<m-1$). Then $f_1=\sigma(U)=ne_1$. Let $x_0$ and $q_2$ be as
defined by (\ref{all-but-one-pi2}). Thus, performing type II swaps
between a fixed $x_1|x_0^{-1}\Wo$ and any $y|V\in \A^*_2\cap
(\C_2\cup \C_0)$, we conclude from $f_1=\sigma(U)=ne_1$ and Lemmas
\ref{lem-Pertebation-II}.2 and \ref{lem-Pertebation-II}.5 that
$\psi_2(V)=\psi_2(x_1)^n$ for all such blocks $V\in \A^*_1\cap
(\C_2\cup \C_0)$. Hence, in view of $ne_1=f_1$, we conclude that
$\pi_2(\psi(V))=\pi_2(\psi(x_1))^n=q_2^n$ for all such $V$, which
combined with (\ref{all-but-one-pi2}) implies CLAIM C. So we may
assume $s=m-1$.

In case $W\in \Omega_0^{nu}$, we have assumed
$\Omega_0^u=\emptyset$. However, we will temporarily drop this
assumption, allowing consideration of $W\in \Omega_0^{nu}$ even when
$\Omega_0^u\neq\emptyset$, provided it still satisfies the hypothesis
of CASE 3 and follows the notation given in the first paragraph
with $s=m-1$. This will extend until the end of assertion \textbf{A1} below, which shows that the exceptional term $x_0$ in (\ref{all-but-one-pi2}) is not necessary.

\smallskip

\noindent
\begin{enumerate}
\item[{\bf A1.}\,] For every $W\in \Omega_0$ satisfying the
hypotheses of CASE 3 (allowing $W\in \Omega_0^{nu}$ even if
$\Omega_0^u\neq \emptyset$), we have $\pi_2(\psi(x_0))=q_2$, where $q_2$ and $x_0$ are as given by (\ref{all-but-one-pi2}).
\end{enumerate}

\smallskip

{\it Proof of \,{\bf A1}}.\,
\  Assume instead there exists $W\in \Omega_0$ satisfying the
hypotheses of CASE 3 with $\pi_2(\psi(x_0))\neq q_2$.

Suppose $x_0|W_j$ with $j>0$. Pull up an arbitrary $y|W_k\in \A_2$, with $k\geq
m$, into a resulting product decomposition $W''$ (such a block exists, else (\ref{all-but-one-pi2}) completes CLAIM C). If $W''$ satisfies the hypotheses of CASE 3, then applying (\ref{all-but-one-pi2}) to $W''$ we conclude that $\pi_2(\psi(y))=q_2$, whence CLAIM C follows in view of (\ref{all-but-one-pi2}) and the arbitrariness of $y$. Therefore we may instead assume $W''$ does not satisfy the hypotheses of CASE 3, whence, in view of CASES 1 and 2, we may assume $W''\in \Omega_0^{nu}$ with $W''_0\in \C_0(W'')$.

Let $z$ be a term dividing both $\Wo$ and ${W''_0}^{(2)}$ (which exists in view of $\sigma(\iota(\Wo))\equiv 1\mod n$). Note that we cannot have $0=\psi(z)-\psi(x_0)+\epsilon(z,x_0)ne_1$, as then
$0=\pi_2(\psi(x_0))-\pi_2(\psi(z))=\pi_2(\psi(x_0))-q_2$, a contradiction to $\pi_2(\psi(x_0))\neq
q_2$. Thus,
in view of (\ref{all-but-one-pi2}) and Lemma \ref{lem-Pertebation-II}.3
or \ref{lem-Pertebation-I}.2, it follows that performing a type II
swap between $x_0|W_j$ and $z|\Wo$ results in a new product
decomposition $W'$ in which $\sigma(W'_j)=Cf_1+f_2$ and
$\sigma(W'_0)=f_1$. Thus, if $W\in \Omega_0^{u}$, then we can apply Lemma \ref{Step-cant-flip-badly} to conclude $W''\in \Omega_0^{u}$, contrary to the conclusion of the previous paragraph. Therefore we may assume $W\in \Omega_0^{nu}$. Hence, from $W''\in \Omega_0^{nu}$ and Lemma \ref{lem-Pertebation-III}, it follows that $\sig(W'')=\sig(W)$, whence $\sigma(W''_0)=f_1+f_2$ (in view of $W''_0\in \C_0(W'')$). However, since $z|{W''_0}^{(2)}$, we may still
apply the previously described swap between $x_0|W''_j=W_j$ and $z|W''_0$ now in $W''$, which results in a product decomposition $W'''\in \Omega'$ with $\mathsf
v_{f_2}(\sig(W'''))=m$ (as $\sigma(W'''_j)=\sigma(W_j')=Cf_1+f_2=f_2$ and $\sigma(W''_j)=\sigma(W_j)=f_1$), contradicting that $S\in \A(G)$.
So we may assume $x_0|W_0$.

Perform a type II swap between an arbitrary $x|\Wo$ and $y|W_j$ with $j\in [1,m-1]$. In view of Lemma \ref{lem-Pertebation-I}.2 or \ref{lem-Pertebation-II}.3, it follows that \be\label{cryingbabe}\epsilon(x,y)ne_1+\psi(x)-\psi(y)\in \{0,\,F\}.\ee If $x=x_0$, then it follows, in view of $\pi_2(\psi(x_0))-\pi_2(\psi(y))=\pi_2(\psi(x_0))-q_2\neq 0$ and (\ref{cryingbabe}), that $\epsilon(x_0,y)ne_1+\psi(x_0)-\psi(y)=F$, and thus \be\label{morecryingbabe} 0\neq \pi_2(\psi(x_0))-q_2=\pi_2(\psi(x_0))-\pi_2(\psi(y))=\pi_2(F).\ee Consequently, if $x\neq x_0$, then, from $\pi_2(\psi(x))-\pi_2(\psi(y))=q_2-q_2=0$ (in view of (\ref{all-but-one-pi2})) and (\ref{cryingbabe}) and (\ref{morecryingbabe}), it follows that $$\epsilon(x,y)ne_1+\psi(x)-\psi(y)=0.$$ As $y|W_j$ with $j\in [1,m-1]$ and $x|x_0^{-1}\Wo$ were arbitrary above, we see that we can apply Lemma \ref{lem-aligment} with $i=0$, $Z=x_0^{-1}\Wo$ and $\Dc=\{W_1,\ldots,W_{m-1}\}$.

Thus we can choose $I$ appropriately so that, for some $q\in
\Ker(\varphi)$, we have that \be\label{tigertiger}\psi(x)=q\ee for
all $x|x_0^{-1}\Wo\prod_{\nu=1}^{m-1}W_{\nu}$, and that \be\label{poof}\iota(x)\leq
\iota(y)\ee for all $x|x_0^{-1}\Wo$ and $y|W_i$, $i\in [1,m-1]$. By
performing a type II swap between $x_0|W_0$ and each $y|W_i$, with
$i\in [1,m-1]$, we conclude, from $\pi_2(\psi(x_0))\neq q_2=\pi_2(q)$ and
either Lemma \ref{lem-Pertebation-I}.2 or
\ref{lem-Pertebation-II}.3, that
\be\label{qutip1}\psi(x_0)-q+\epsilon(x_0,y)ne_1=(C-1)f_1+f_2.\ee
Thus $\epsilon(x_0,y)$ must be the same for every $y|W_j$ with $j\in [1,m-1]$. As a result,
it follows in view of (\ref{align-eps-choice}) that either
$\iota(x_0)\leq \min(\supp(\iota( \prod_{\nu=1}^{m-1}W_{\nu})))$ or
$\iota(x_0)>\max(\supp(\iota(\prod_{\nu=1}^{m-1}W_{\nu})))$. In the latter
case, we may choose $I$ such $\min I\equiv \iota(x_0)\mod n$, and
thus, in both cases,  we have (in view of (\ref{poof})) \be\label{ordering}\iota(x)\leq
\iota(y)\ee for all $x|\Wo$ and $y|W_i$, $i\in [1,m-1]$, while still
preserving that (\ref{tigertiger}) holds for some $q\in
\Ker(\varphi)$ (since (\ref{ordering}) was all that was required in the proof of Lemma \ref{lem-aligment} to ensure (\ref{tigertiger}) held). Consequently, (\ref{qutip1}) and (\ref{align-eps-choice}) imply that
\be\label{psivalu-xo}\psi(x_0)=q+F=q+(C-1)f_1+f_2.\ee

Let $y|W_k\in \A_2$ with $k\geq m$ and $\pi_2(\psi(y))\neq q_2$;
such a term and block exists else  CLAIM C follows in view of
(\ref{all-but-one-pi2}). If $y|W_k$ could be pulled up into a new
product decomposition $W'$ with $x_0|W'_0$, then $W'$ must still satisfy the
hypothesis of CASE 3 (by the same arguments used when $x_0|W_j$ with $j>0$), whence applying (\ref{all-but-one-pi2}) to
$W'$ implies $\pi_2(\psi(x_0))=q_2$ or $\pi_2(\psi(y))=q_2$, contrary to our assumption.
Therefore we may assume this is not the case, whence Theorem
\ref{EGZ-thm}.2 implies that \be\label{thelist}\iota(\Wo)=g_1^lg_2^{n-1-l}\iota(x_0)\mbox{
and }\iota(W_k)=g_1^{n-1-l}g_2^l\iota(y),\ee for some $g_1,\,g_2\in
\Z$ with $\gcd(g_1-g_2,n)=1$. If there existed $x'_0|\Wo$ such that
$\epsilon(x'_0,z)=\epsilon(x_0,z)$ for some $z|W_k$, then we could
apply a type II swap between $z|W_k$ and each of $x_0|W_0$ and
$x'_0|W_0$, which in view of Lemma \ref{lem-Pertebation-I}.3 or
Lemma \ref{lem-Pertebation-II} would imply that
$\psi_2(x_0)=\psi_2(x'_0)=\psi_2(q)$, contradicting (\ref{psivalu-xo}).
Therefore we may assume otherwise, whence (\ref{align-eps-choice}) implies either \be\iota(x_0)\leq
\min(\supp(\iota(W_k))\leq
\max(\supp(\iota(W_k))<\min(\supp(\iota(x_0^{-1}\Wo))\label{zup1}\ee
or \be\iota(x_0)> \max(\supp(\iota(W_k))\geq
\min(\supp(\iota(W_k))\geq
\max(\supp(\iota(x_0^{-1}\Wo))\label{zup2}.\ee
In either case, we see that $|\supp(\iota(W_k))\cap \supp(\iota(\Wo))|\leq 1$. As a result, (\ref{thelist}) implies that w.l.o.g. $l=n-1$, $\iota(\Wo)=g_1^{n-1}\iota(x_0)$ and
$\iota(W_k)=g_2^{n-1}\iota(y)$. Thus $\sigma(\iota(W_k))\equiv
0\mod n$ and $\sigma(\iota(\Wo))\equiv 1\mod n$ imply that
$\iota(W_k)=g_2^{n}$ and $\iota(x_0)\equiv g_1+1\mod n$.

If (\ref{zup1}) holds, then from $\iota(x_0)\equiv g_1+1\mod n$ and (\ref{zup1}) it follows that $\max I\equiv g_1\mod n$. However, in view
of (\ref{ordering}), this is only possible if $\iota(x)\equiv
g_1\mod n$ for all $x|x_0^{-1}\Wo\prod_{\nu=1}^{m-1}W_{\nu}$, in which
case, since $\psi(x)=q$ also holds for all such terms (in view of (\ref{tigertiger})), it follows
that $S$ contains a term with multiplicity $mn-1$, as desired. Therefore we can instead assume
(\ref{zup2}) holds. In this case, it follows, in view of (\ref{zup2}), $\iota(x_0^{-1}\Wo)=g_1^{n-1}$ and $\iota(x_0)\equiv g_1+1\mod n$, that $$\{g_2\}=\supp(\iota(W_k))=\supp(\iota(x_0^{-1}\Wo))=\{g_1\},$$ contradicting that  $\gcd(g_1-g_2,n)=1$.
\qedsymbol

\smallskip

We now return to arguments where we assume $\Omega_0^u=\emptyset$
when $W\in \Omega_0^{nu}$. In view of \textbf{A1}, we may assume
$\pi_2(\psi(x))=q_2$ for all $x|\Wo\prod_{\nu=1}^{m-1}W_{\nu}$. Let
$y|W_k$, with $W_k\in \A_2$ and $k\geq m$, be arbitrary. If we can
pull up $y$ into a new product decomposition $W'$ such that either
$W'\in \Omega_0^u$, or else $W'\in \Omega_0^{nu}$ and $W'_0\notin \C_0(W')$, then it follows, in view of CASES 1 and 2, \textbf{A1} and (\ref{all-but-one-pi2}), that we may assume
$\pi_2(\psi(y))=q_2$ also (note this is where we need that $W\in \Omega_0^{nu}$ is allowed in \textbf{A1} even when $\Omega_0^u\neq \emptyset$). However, this can only fail if (by an appropriate choice for $f_2$ in the case when $W\in \Omega_0^u$) w.l.o.g.
\be\label{thebadasssequence}\sig(W)=f_1^{m-1}f_2^{m-2}(Cf_1+f_2)((1-C)f_1+f_2),\ee with
$\sigma(W_k)=(1-C)f_1+f_2$ and (recall) $\sigma(W_0)=Cf_1+f_2$. Consequently, we see that there is at
most one block $W_k$ for which this can fail (as $W_0\notin \C_0$
when $\Omega_0^{u}=\emptyset$). As CLAIM C follows otherwise, we may
assume $W_k\in \A_2$ exists and that $\sig(W)$ is of such form, and w.l.o.g.
assume $k=2m-2$. Then
\ber\label{uber1}Cf_1+f_2=\sigma(W_0)=Y_1ne_1+ne_2+nq_2,\\
\label{uber2} f_1=\sigma(W_1)=Y_2ne_1+ne_2+nq_2,\eer for some $Y_i\in \Z$. From (\ref{uber1}) and
(\ref{uber2}), we conclude that \be\label{uber-off}(C-1)f_1+f_2\in
\langle ne_1 \rangle .\ee

If there exists $U\in \A_1^*$, then $ne_1=\sigma(U)=f_2$ (in view of (\ref{thebadasssequence}), $s=m-1$ and $W_k=W_{2m-2}\in \A_2$); thus from  (\ref{uber-off}) it follows that $(C-1)f_1\in \langle f_2 \rangle$, which is
only possible if $C\equiv 1\mod m$, contradicting that $W\notin
\C_0$ when $W\in\Omega_0^{nu}$ (in view of (\ref{thebadasssequence})). So we may instead assume $|\A_1|=1$. This same argument also shows that $\psi_1(ne_1)\neq 0$. Let $\Dc=\{W_m,\ldots,W_{2m-2}\}$.

If $\psi_2(ne_1)=0$, then $ne_1\in \langle f_1\rangle$, which combined with (\ref{uber-off}) yields a contradiction to $(f_1,f_2)$ being a basis. Therefore $\psi_2(ne_1)\neq 0$. Thus, in view of Lemma \ref{lem-Pertebation-I}.3 or Lemmas \ref{lem-Pertebation-II}.5 and \ref{lem-Pertebation-II}.2, it follows that we may apply Lemma \ref{lem-aligment} with $Z=\Wo$, $i=2$ and $\Dc$ as given above. Choose $I$ as directed by Lemma \ref{lem-aligment} (as mentioned before, changing $I$ does not affect the value of $\pi_2(\psi(x))$, and thus (\ref{all-but-one-pi2}) remains unaffected). Then \be\label{psi2-constant}\psi_2(x)=\alpha_2,\ee for all $x|\Wo\prod_{\nu=m}^{2m-2}W_{\nu}$ and some $\alpha_2\in \langle f_2\rangle$, and \be\label{lined-up}\iota(x)\leq \iota(y),\ee for all $x|\Wo$ and $y|\prod_{\nu=m}^{2m-2}W_{\nu}$.

Let $y_0|W_{2m-2}$ with $\pi_2(\psi(y_0))\neq q_2$ (such $y_0$
exists, as discussed above, else CLAIM C follows). Let $W'$ be an
arbitrary product decomposition resulting from pulling up $y_0$ into
a new product decomposition. Since $\pi_2(\psi(y_0))\neq q_2$, we
have (as discussed earlier) $\sig(W')=f_1^{m-1}f_2^{m-1}(f_1+f_2)$
with $\sigma(W'_0)=f_1+f_2$. Let $X=\gcd(\Wo,{W'_0}^{(2)})$ and let
$X'$, $Y'$ and $Y$ be defined by $\Wo=XX'$, ${W'_0}^{(2)}=XY'$ and
$W_{2m-2}=YY'$. Thus $W'_{2m-2}=X'Y$. Note that all four of these
newly defined subsequences are nontrivial in view of
$\sigma(\iota(\Wo))\equiv 1\mod n$ and
$\sigma(\iota(W_{2m-2}))\equiv 0\mod n$.

Let $\Dc'=\{W'_0,W'_1,\ldots,W'_{m-1}\}$. In view of Lemma
\ref{lem-Pertebation-II}.4  and $\psi_1(ne_1)\neq 0$, it follows
that we can apply Lemma \ref{lem-aligment} with $i=1$,
$Z={W'_0}^{(2)}$, and $\Dc$ taken to be $\Dc'$ (however, do NOT
change $I$). If (\ref{vogue2}) holds, then (in view of (\ref{align-eps-choice})) we can find $z|W'_j$, for
some $j\in [1,m-1]$, such that $\epsilon(y_0,z)=\epsilon(x,z)$,
where $x|X$. Applying a type II swap between $z|W'_j$ and each of
$x|W'_0$ and $y_0|W'_0$, we conclude from Lemma
\ref{lem-Pertebation-II}.4 that $\psi_1(x)=\psi_1(y_0)$. However,
since $x|X$ and $X|\Wo$, it follows from (\ref{psi2-constant}) that
$\psi_2(x)=\psi_2(y_0)$ also, whence $\psi(x)=\psi(y_0)$, implying
$q_2=\pi_2(\psi(x))=\pi_2(\psi(y_0))$, contrary to assumption.
Therefore we may instead assume (\ref{vougue1}) holds. Moreover, if
both $y_0$ and some $x|X$ are contained in the same interval $J_i$
(from (\ref{vougue1})), then we can repeat the above argument to
obtain the same contradiction. Therefore it follows, in view of
(\ref{lined-up}), that $y_0\in J_2$ and $X\subset J_1$.

Let $z|{W'_0}^{(2)}$ and $z'|W'_j$ with $j\geq m$ be arbitrary. Performing a type II swap between $z|{W'_0}^{(2)}$ and $z'|W'_j$, we conclude from Lemma \ref{lem-Pertebation-II}.5 that $$\psi_2(z)-\psi_2(z')+\psi_2(\epsilon(z,z')ne_1)=0.$$ Thus (\ref{psi2-constant}) implies that $\psi_2(\epsilon(z,z')ne_1)=0$, which, in view of $\psi_2(ne_1)\neq 0$ and (\ref{align-eps-choice}), implies that $\epsilon(z,z')=0$ and \be\label{second-aligning}\iota(z)\leq \iota(z'),\ee for any $z|{W'_0}^{(2)}$ and $z'|W'_j$ with $j\geq m$.

Applying (\ref{second-aligning}) using $z|Y'$ and $z'|X'$ and $j=2m-2$, we conclude in view of (\ref{lined-up}) that \be\label{loudly}\iota(z)=\max (\supp(\iota({W'_0}^{(2)})))=\min(\supp(\iota(\prod_{\nu=m}^{2m-2}W'_{\nu})))=\iota(z'),\ee for any $z'|X'$ and $z|Y'$.

From (\ref{loudly}) applied with $z=y_0$, we see that there is
$y'_0|\Wo$ with $\iota(y'_0)=\iota(y_0)$. Thus $y$ can be pulled up
into a new decomposition $W''$ by exchanging $y_0|W_{2m-2}$ and
$y'_0|W_0$, and all of the above arguments (valid for an arbitrary
$W'$ obtained by pulling up $y_0|W_{2m-2}$) are applicable for
$W''$. In particular, ${y'_0}^{-1}\Wo=X\subset J_1$ and $y_0\in J_2$ imply, in view
of $Y=y_0^{-1}W_{2m-2}$, (\ref{vougue1}) and (\ref{loudly}),
that
\be\label{slowly}\max(\supp(\iota({y'_0}^{-1}\Wo)))<\min(\supp(\iota(W_{2m-2}))).\ee

If we could pull up $y'_0y_0|W_0W_{2m-2}$ into a new product
decomposition $W'''$, then (\ref{slowly}) would imply that $X'$
contains a $z'$ with $\iota(z')<\iota(y_0)$, which would contradict
(\ref{second-aligning}) applied with $z=y_0$ and $z'=z'$. Therefore
we can assume otherwise, whence Theorem \ref{EGZ-thm}.2 and
(\ref{slowly}) imply that
$|\supp(\iota({y'_0}^{-1}\Wo))|=|\supp(\iota(y_0^{-1}W_{2m-2}))|=1$.
Thus $\sigma(\iota(\Wo))\equiv 1\mod n$ and
$\sigma(\iota(W_{2m-2}))\equiv 0\mod n$ force that
$\iota(W_{2m-2})=g^n$ and $\iota(\Wo)=(g-1)^{n-1}g$, where
$\iota(y_0)=\iota(y'_0)=g$. Consequently, (\ref{vougue1}), $X\subset
J_1$ and $y_0\in J_2$ (in the case when $W'=W''$) force that
$\iota(z)=g$ for all $z|{y'_0}^{-1}\Wo\prod_{\nu=1}^{m-1}W_i$.

Applying type III swaps among the $W_i$, $i\in [1,m-1]$, we conclude from Lemma \ref{lem-Pertebation-III}.1 or \ref{lem-Pertebation-I}.1 that $\psi(x)=q$ (say) for all $x|W_i$, $i\in [1,m-1]$. Applying type III swaps between $W_0$ and $W_1$, we conclude from Lemma \ref{lem-Pertebation-II}.3 or \ref{lem-Pertebation-I}.2 and Lemma \ref{lem-basic-exchange}.3 that $\psi(x)=q$ for all $x|{y''_0}^{-1}{y'_0}^{-1}\Wo$, for some $y''_0|{y'_0}^{-1}\Wo$, and that $\psi(y''_0)=q$ or $q+(C-1)f_1+f_2$. Applying a type III swap between $y''_0|W''_0$ and some $z|W''_1$ in $W''$, we conclude from Lemma \ref{lem-Pertebation-II}.4 that $\psi_1(y''_0)=\psi_1(z)=\psi_1(q)$, whence we see that $\psi(y''_0)=q+(C-1)f_1+f_2$ is impossible (since $C\equiv 1\mod m$ would contradict that $W_0\notin \C_0$ when $W\in \Omega_0^{nu}$; see (\ref{thebadasssequence})). Thus $\psi(y''_0)=q$ as well, and $ge_1+e_2+q$ has multiplicity at least $mn-1$ in $S$, as desired, completing CASE 3.

\smallskip

\noindent CASE 4: \ $\Omega_0^u=\emptyset$ \ and \ $W_0\in \C_0$.

We start with the following assertion.

\begin{enumerate}
\item[{\bf A2.}\,] If $\Omega_0^u=\emptyset$, $W\in \Omega_0^{nu}$ with
$\sig(W)=f_1^{m-1}f_2^{m-1}(f_1+f_2)$, $W_0\in \C_0$, and $|\A_2\cap
\C_i|\geq 1$ for all $i \in \{1,2\}$, then $I$ can be chosen such
that one of the following properties holds{\rm \,:}
\begin{enumerate}
\item[(i)] $|\supp(\psi(\Wo))|=1$, or

 \smallskip
\item[(ii)]
\begin{enumerate}
\item[(a)] $\psi_i(ne_1)\neq 0$ for all $i \in \{1,2\}$,

\item[(b)] there exist $g_1,\,g_2\in \Z$ such that
           $\gcd(g_1-g_2,n)=1$ and $\iota(U)=g_1^n$ and $\iota(V)=g_2^n$, for
           every $U\in \A^*_2\cap \C_1$ and $V\in \A^*_2\cap \C_2$,

\item[(c)] $g_1>g_2$ and $\iota(x)\leq g_1$ for all $x|\Wo$, and
\item[(d)] if also $|\A_2\cap \C_i|\geq 2$ for all $i \in \{1,2\}$, then there exist $c,\,d\in \Ker(\varphi)$ such that $\psi(U)=c^n$ and
                   $\psi(V)=d^n$ for every $U\in \A^*_2\cap \C_1$ and $V\in
                   \A_2^*\cap\C_2$.
\end{enumerate}
\end{enumerate}
\end{enumerate}

\smallskip

{\it Proof of \,{\bf A2}}.\, We may w.l.o.g. assume $\C_1$ are those blocks with sum $f_1$. Performing type II swaps between each
$x|\Wo$ and each $y|U\in \A^*_2\cap \C_1$, and between each $x|\Wo$
and each $z|V\in \A^*_2\cap \C_2$, we conclude from Lemma
\ref{lem-Pertebation-II} that
\ber\label{pseudo-start-1}\psi_1(x)&=&\psi_1(y)-\psi_1(\epsilon(x,y)ne_1),\\
\psi_2(x)&=&\psi_2(z)-\psi_2(\epsilon(x,z)ne_1)\label{pseudo-start-2},\eer
where (\ref{align-eps-choice}) holds.

Since $\ord(e_1)=mn$, one of $\psi_1(ne_1)$ or $\psi_2(ne_1)$ is nonzero, say the former
(the other case will be identical). Then, in view of
(\ref{pseudo-start-1}), we may apply Lemma \ref{lem-aligment} with $i=1$, $Z=\Wo$ and $\Dc=\A_2^*\cap\C_1$. Consequently, we can choose $I$ such that \be\label{poobear}\iota(x)\leq \iota(y),\ee
for all $x|\Wo$ and $y|U\in \A^*_2\cap \C_1$, and $\psi_1$ is constant on $\Wo$.
If $\psi_2(ne_1)$ is zero, then (\ref{pseudo-start-2}) implies that
$\psi_2$ is also constant on $\Wo$, whence (i) holds. Therefore we
may assume otherwise, and (a) is established. Likewise, if there is
some $z|V\in \A^*_2\cap \C_2$ with $\iota(z)\geq \max
(\supp(\iota(\Wo))$ or $\iota(z)<\min(\supp(\iota(\Wo)))$, then (i) again holds (in view of (\ref{align-eps-choice}) and (\ref{pseudo-start-2})). So we may assume
otherwise:
\be\label{uxed} \min(\supp(\iota(\Wo)))\leq \iota(z)<\max(\supp(\iota(\Wo))),\ee for all $z|V\in \A_2^*\cap \C_2$.
Consequently, it follows in view of (\ref{poobear}) that
both $\supp(\iota(\prod_{U\in \A^*_2\cap \C_1}U))$ and
$\supp(\iota(\prod_{V\in \A^*_2\cap \C_2}V))$ are disjoint.

Suppose $|\supp(\iota(U))|>1$ or $|\supp(\iota(V))|>1$, for some $U\in
\A^*_2\cap \C_1$ or $V\in \A^*_2\cap \C_2$. Then we may find $u_0|U$ and $v_0|V$ such that $|\supp(\iota(u_0^{-1}U))|>1$ or $|\supp(\iota(v_0^{-1}V))|>1$, whence it follows, in view of Theorem \ref{EGZ-thm}.2 (applied to $\iota(u_0^{-1}v_0^{-1}UV)$ modulo $n$) and the fact that $\supp(\iota(\prod_{U\in \A^*_2\cap \C_1}U))$ and
$\supp(\iota(\prod_{V\in \A^*_2\cap \C_2}V))$ are
disjoint, that we can refactor $UV=U'V'$ such that $U'$ and $V'$ both contain
terms from both $U$ and $V$. Replacing the blocks $U$ and
$V$ by the blocks $U'$ and $V'$ yields a new product decomposition
$W'\in \Omega_0$; in view of Lemma \ref{lem-Pertebation-II}.3, we
still have $\sig(W')=\sig(W)$, whence $W'$ satisfies the hypotheses
of \textbf{A2}. However, since both $U'$ and $V'$ contain terms from
both $U$ and $V$, it follows that both $U'$ and $V'$ contain a term $z'|U$ with $\iota(z')\geq \max(\supp(\iota(\Wo)))$ (in view of (\ref{poobear})), as well as a term $z|V$ with $\min(\supp(\iota(\Wo)))\leq \iota(z')< \max(\supp(\iota(\Wo)))$ (in view of (\ref{uxed})), which makes it impossible for (\ref{vougue1}) or (\ref{vogue2}) to hold for $W'$, contradicting that the above arguments show Lemma \ref{lem-aligment} must hold for $W'$.  So we may assume $|\supp(\iota(U))|=1$ and
$|\supp(\iota(V))|=1$ for all $U\in \A^*_2\cap \C_1$ and $V\in
\A^*_2\cap \C_2$. Moreover, this argument also shows that if $\iota(U)=g_1^n$
and $\iota(V)=g_2^n$, then $\gcd(g_1-g_2,n)=1$.

Suppose $|\supp(\iota(\prod_{U\in \A^*_2\cap
\C_1}U))|>1$ or $|\supp(\iota(\prod_{V\in \A^*_2\cap \C_2}V))|>1$, say the former (the other case will be identical). Then there are $U_1,\,U_2\in \A_2^*\cap \C_1$ and $V\in \A_2^*\cap \C_2$ with $\iota(U_1)=g_1$, $\iota(U_2)=g'_1$ and $\iota(V)=g_2$, where $g_1\neq g'_1$.
We have $\gcd(g_1-g'_1,n)=1$, else repeating the arguments of the previous paragraph, using $U_1$ and $U_2$ in place of $U$ and $V$, we obtain a $W'\in \Omega_0$ satisfying the hypotheses of \textbf{A2} but such that the conclusion of the previous paragraph fails, whence $1=|\supp(\psi({W'_0}^{(2)}))|=|\supp(\psi({W_0}^{(2)}))|$ must hold by prior arguments, yielding (i). Hence, since $\gcd(g_1-g_2,n)=1$ and $\gcd(g'_1-g_2,n)=1$, it follows that all $n$-term zero-sum modulo $n$ subsequences of $g_1^{n-1}{g'_1}^{n-1}g_2^{n-1}$ have support of cardinality three.
Thus, by two applications of Theorem \ref{EGZ-thm}.1, we see that we can refactor $U_1U_2V=XYZ$ such that $X$, $Y$ and $Z$ all contain terms from each of $U_1$, $U_2$ and $V$ (note, since $|\supp(\iota(X))|=3$, that $\iota(YZ)\subset g_1^{n-1}{g'_1}^{n-1}g_2^{n-1}$). Replacing $U_1$, $U_2$ and $V$ by $X$, $Y$ and $Z$ yields a new product decomposition
$W'\in \Omega_0$; in view of $\Omega_0^u=\emptyset$ and $m\geq 5$, we
still have $\sig(W')=\sig(W)$, whence $W'$ satisfies the hypotheses
of \textbf{A2}. However, since  $X$, $Y$ and $Z$ each contain terms from  $U_1$, $U_2$ and $V$,
we see that the condition
$|\supp(\iota(U))|=1$ for $U\in \A_2^*\cap \C_1$ fails for $W'$, whence previous arguments
show $|\supp(\psi(\Wo))|=|\supp(\psi({W_0'}^{(2)}))|=1$, yielding
(i). So we may assume $|\supp(\iota(\prod_{U\in \A^*_2\cap
\C_1}U))|=1$ and $|\supp(\iota(\prod_{V\in \A^*_2\cap \C_2}V))|=1$, and w.l.o.g. assume $\supp(\iota(\prod_{U\in \A^*_2\cap
\C_1}U))=g_1$ and $\supp(\iota(\prod_{V\in \A^*_2\cap \C_2}V))=g_2$.
This establishes (b). Moreover, by the arguments from the second paragraph, we see that we can choose $I$ such that (c) holds.

We now assume $|\A_2\cap \C_i|\geq 2$, for all $i \in \{1,2\}$.
Performing type III swaps between distinct $U_1,\,U_2\in \A^*_2\cap
\C_1$ and between distinct $V_1,\,V_2\in \A^*_2\cap \C_2$, we
conclude from Lemma \ref{lem-Pertebation-III} that $\psi(U)=c$ (say)
for all $U\in \A^*_2\cap \C_1$ and that $\psi(U)=d$ (say) for all
$V\in \A^*_2\cap \C_2$, establishing (d), and completing the proof
of \textbf{A2}.\qedsymbol

\smallskip

Since $\Omega_0^u=\emptyset$, it follows, in view of Lemma
\ref{lem-Pertebation-III}, that if we pull up any term $y|U$, where
$U\in \A^*_2$, then we may assume the resulting product
decomposition still satisfies the hypothesis of CASE 4, else CASE
3 completes the proof. Thus, if for every product decomposition
satisfying the hypothesis of CASE 4 we can find $I$ such that
$|\supp(\psi(\Wo))|=1$, then, since modifying $I$ does not alter the
values $\pi_2(\psi(x))$, we would be able to conclude
$|\supp(\pi_2(\psi(S_2)))|=1$---by successively pulling up terms
$y|S_2$, yielding a sequence of product decompositions satisfying
the hypotheses of CASE 4, until every such $y$ occurred in the
$\Wo$ part of one of these product decompositions, and then noting
that there must always be a common term in $\Wo$ between any two
consecutive product decompositions in the sequence (in view of $\sigma(\iota(\Wo))\equiv 1\mod n$)---completing CLAIM C. Therefore we may assume this is not the case for $W$. Let w.l.o.g. $\sig(W)=f_1^{m-1}f_2^{m-1}(f_1+f_2)$ and $\C_1$ consist of those blocks with sum $f_1$.

Note that we must have $\A_2^*\cap \C_1$ and $\A_2^*\cap \C_2$ both
nonempty, else in view of CLAIM B it would follow that $e_1$ is a
term of $S$ with multiplicity $mn-1$, completing the proof. Thus
{\bf A2.(ii)(a)} implies that $\psi_i(ne_1)\neq 0$ for $i \in \{
1,2\}$. As a result, we cannot have a block $U\in \A_1^*$ (else $ne_1=\sigma(U)=f_1$ or $f_2$). Hence $|\A_1|=1$, implying $|\A_2^*\cap \C_1|\geq 2$
and $|\A_2^*\cap \C_2|\geq 2$. Thus, by choosing $I$ appropriately, {\bf A2.(ii)(a--d)} holds for
$W$.

Suppose $\supp(\iota(\Wo))\neq \{g_1,\,g_2\}$. Then there must be
some $x_0|\Wo$ with $\iota(x_0)\notin\{g_1,\,g_2\}$ (in view of
$\sigma(\iota(\Wo))\equiv 1\mod n$). Since $\gcd(g_1-g_2,n)=1$,
there is no $n$-term zero-sum mod $n$ subsequence of
$g_1^{n-1}g_2^{n-1}$. Thus applying Theorem \ref{EGZ-thm}.1 to
$g_1^{n-1}g_2^{n-1}\iota(x_0)$ implies that we may find a
subsequence $U_1|\Wo UV$, where $U\in \A_2^*\cap \C_1$ and $V\in
\A_2^*\cap \C_2$, such that $x_0|U_1$ and
$\supp(\iota(z^{-1}U_1))=\{g_1,\,g_2\}$. Consequently,
$\vp_{g_i}(U_1)\leq n-2$, and thus $\vp_{g_i}(\iota(U_1^{-1}\Wo
UV))\geq 2$, for $i = \{1,2\}$. Thus, if there were no $n$-term
zero-sum mod $n$ subsequence of $\iota(U_1^{-1}u_1^{-1}v_1^{-1}\Wo
UV)$, where $u_1|U_1^{-1}U$ and $v_1|U_1^{-1}V$, then Theorem
\ref{EGZ-thm}.2 would imply that $\iota(U_1^{-1}\Wo UV)=g_1^ng_2^n$,
whence $$1\equiv \sigma(\iota(\Wo UV))\equiv \sigma(\iota(U_1))+ng_1+ng_2\equiv 0\mod n,$$ which is a
contradiction. Therefore we may assume there exists such a
subsequence $\iota(U_2)$, where $U_2|U_1^{-1}u_1^{-1}v_1^{-1}\Wo UV$.
Let $W'_0$ be defined by $W_0UV=U_1U_2W'_0$. Then replacing the
blocks $W_0$, $U$ and $V$ with the blocks $W'_0$, $U_1$, and $U_2$
yields a new product decomposition $W'\in \Omega_0$. Since
$\Omega_0^u=\emptyset$ and $m\geq 4$, we must have $\sig(W)=\sig(W')$, and we may
further assume $W'_0\in \C_0$ else CASE 3 completes the proof. Thus
$W'$ satisfies the hypotheses of CASE 4, but since
$|\supp(\iota(U_1))|>1$, we see that $W'$ does not satisfy {\bf
A2.(ii)}. Thus \textbf{A2.(i)} implies that we must have $|\supp(\pi_2(\psi({W'_0}^{(2)}))|=1$ (note we do not have $|\supp(\psi({W'_0}^{(2)})|=1$ as we would need to change $I$ for this to hold); since
$u_1v_1|W'_0$, this implies that
$\pi_2(c)=\pi_2(\psi(u_1))=\pi_2(\psi(v_1))=\pi_2(d)$.

Let $x|x_0^{-1}\Wo$ be arbitrary. By Theorem \ref{EGZ-thm}.1, it follows that
there is an $n$-term zero-sum mod $n$ subsequence of
$\iota(x^{-1}U_1^{-1}\Wo UV)$, say $\iota(U_3)$ with
$U_3|x^{-1}U_1^{-1}\Wo UV$. Let $W''_0$ be defined by
$W_0UV=U_1U_3W''_0$. Then replacing the blocks $W_0$, $U$ and $V$
with the blocks $W''_0$, $U_1$, and $U_3$ yields a new product
decomposition $W''\in \Omega_0$, and as before we may assume $W''$
satisfies the hypotheses of CASE 4 with $\sig(W'')=\sig(W)$. Thus, since
$|\supp(\iota(U_1))|>1$, we see that $W''$ does not satisfy {\bf
A2.(ii)}, and so we must have \be\label{laughingcow}|\supp(\pi_2(\psi({W_0''}^{(2)}))|=1.\ee Since
$x_0|U_1$, it follows in view of the pigeonhole principle that we
must have a term $x'|{W_0''}^{(2)}$ with $x'|UV$, and thus with
$\pi_2(\psi(x'))=\pi_2(c)$ (in view of the previous paragraph).
Since $x|W''_0$, this implies $\pi_2(\psi(x))=\pi_2(c)$ (in view of (\ref{laughingcow})). As
$x|x_0^{-1}\Wo$ was arbitrary, we conclude that every
$x|x_0^{-1}S_2$ has $\pi_2(\psi(x))=\pi_2(c)=\pi_2(d)$, completing
the proof (in view of \textbf{A2.(ii)} holding for
$W$). So we may instead assume  $\supp(\iota(\Wo))= \{g_1,\,g_2\}$.

Since $|\A_1|=1$, let w.l.o.g. $W_1,\ldots,W_{m-1}$ be the blocks of
$\A^*_2\cap \C_1$, and let $W_m,\ldots,W_{2m-2}$ be the blocks of
$\A^*_2\cap \C_2$. Let $\Wo = b_1 \cdot \ldots \cdot b_t b'_1 \cdot
\ldots \cdot b'_{n-t}$ with $\iota(b_i)=g_1$ and $\iota(b'_j)=g_2$.
Applying type III swaps between $b_i|W_0$ and $y|W_1$, it follows from Lemma \ref{lem-Pertebation-III}.4
that we may assume
$\psi(b_i)=\psi(y)=c$ for all $i$ (else CASE 3 completes the proof).
Likewise applying type III swaps between $b'_i|W_0$ and $z|W_m$, it
follows that $\psi(b'_i)=\psi(z)=d$ for all $i$. Consequently, we
may assume $t\in [2, n-2]$, else $S$ contains a term with
multiplicity at least $mn-1$, as desired (either $g_1e_1+e_2+c$ or
$g_2e_1+e_2+d$).

Applying type II swaps between $b_1|W_0$ and $z|W_m$ and between
$b'_1|W_0$ and $y|W_1$, it follows, in view of Lemma
\ref{lem-Pertebation-II}, (\ref{align-eps-choice}) and $g_1>g_2$, that \ber\label{uz1} d-c\in
\langle f_2 \rangle, \\ c-d+ne_1\in \langle f_1 \rangle
\label{uz2}.\eer Since $t \in [2, n-2]$, we have $b_1b_2|\Wo$ and
$b'_1b'_2|\Wo$. Let $Y$ be a subsequence of $W_1$ and  $Z$ be a
subsequence of $W_m$ with $|Y| = |Z| = 2$.  Applying type II swaps
between $b'_1b'_2|W_0$ and $Y|W_1$ and between $b_1b_2|W_0$ and
$Z|W_m$, we conclude from Lemma \ref{lem-Pertebation-II} that \ber
2(d-c)+\epsilon(b'_1b'_2,Y)ne_1\in \langle f_2 \rangle,
\label{uzz1}\\\label{uzz2}2(c-d)+\epsilon(b_1b_2,Z)ne_1\in \langle
f_1 \rangle .\eer Observe (in view of $g_1>g_2$) that $$\epsilon(b'_1b'_2,Y)ne_1=\left\{
                             \begin{array}{ll}
                               0, & \hbox{if } g_1-g_2\leq \frac{n-1}{2}; \\
                               -ne_1, & \hbox{if } g_1-g_2\geq \frac{n+1}{2}.
                             \end{array}
                           \right.$$ Likewise
$$\epsilon(b_1b_2,Z)ne_1=\left\{
                             \begin{array}{ll}
                               ne_1, & \hbox{if } g_1-g_2\leq \frac{n-1}{2}; \\
                               2ne_1, & \hbox{if } g_1-g_2\geq \frac{n+1}{2}.
                             \end{array}
                           \right.$$
Thus, if $g_1-g_2\leq \frac{n-1}{2}$, then (\ref{uzz2}) and
(\ref{uz2}) imply that $c-d\in \langle f_1 \rangle$, which combined
with (\ref{uz1}) implies that $c=d$, in which case CLAIM C follows. On the other hand, if $g_1-g_2\geq \frac{n+1}{2}$, then
(\ref{uzz1}) and (\ref{uz1}) imply that $ne_1\in \langle f_2
\rangle$, which contradicts that {\bf A2.(ii)(a)} holds for $W$,
completing CASE 4.
\end{proof}

\medskip
\centerline{{\bf CLAIM D:} \ $\mathsf h(S)=mn-1$.}
\medskip

\begin{proof}
Let $S'_2=x_0^{-1}S_2$, with $x_0$ as in CLAIM C, and let $S'=S_1S'_2$. By
Proposition \ref{mainproposition} and CLAIM B, we have $S_1 =
{e_1}^{|S_1|}$, $|S_1|= \ell n-1$ and $|S'_2|=2mn - \ell n-1$, for
some $\ell \geq 1$. If $\ell \geq m$, then $e_1$ is a term with
multiplicity at least $mn-1$, as desired. Therefore $\ell < m$.
Moreover, since $S\in \A(G)$, it follows that $0\notin \Sigma(S')$.
In view of CLAIM C, we may assume every $x|S'_2$ is of the form
$y_ie_1+(1+nq)e_2$, with $q\in [0,m-1]$. Let  $T = \pi_1(S'_2) \in \Fc( \langle e_1 \rangle)$, and
let $H'=\langle e_1,\,(1+qn)e_2\rangle \cong C_{mn}\oplus C_{rn},$
where $rn=\ord((1+qn)e_2)$. If $r<m$, then noting that $S'\in
\Fc(H')$ with $|S'|=2mn-2\geq mn+rn-1 = \mathsf D(H')$, we see that
$0\in \Sigma(S')$, contradicting that $S\in \A(G)$. Thus we may
choose $e_2$ to be $(1+qn)e_2$ while still preserving that
$(e_1,e_2)$ is a basis, and so w.l.o.g. we assume $q=0$.

Since $\ell<m$, it follows that $|S'_2|=2mn - \ell n-1\geq
mn+n-1\geq mn+2$ and
\be\label{l-not-m-case}\Sigma(S_1)=\{e_1,\,2e_1,\ldots, (\ell
n-1)e_1\}.\ee Consequently, $0\notin \Sigma(S')$ implies
\be\label{subset-for-S2} \Sigma_{mn}(S'_2)=\Sigma_{mn}(T)\subset
A:=\{e_1,\,2e_1,\,\ldots,(mn - \ell n)e_1\},\ee and thus
\be\label{card-bound} |\Sigma_{mn}(T)|\leq mn - \ell n=|T|-mn+1.\ee
Note $\mathsf h (T) = \mathsf h (S'_2)\leq mn-2$, else the proof is
complete. Thus we can apply Theorem \ref{ham-result}, taking $k=3$,
whence it follows, in view of (\ref{card-bound}) and $0\notin
\Sigma_{mn}(T)$, that $|\supp(T)|\leq 2$.

We may assume $|\supp(T)|=2$, else $S$ will contain a term with
multiplicity $|T| = 2mn - \ell n-1\geq mn+n-1$, contradicting that
$S\in \A(G)$. Thus $T=(g_0e_1)^{n_1}((g_0+d)e_1)^{n_2}$ for some
$g_0,\,d\in\Z$ with $de_1\neq 0$. Since $(e_1,g_0e_1+e_2)$ is also a basis
for $G$, then, by redefining $e_2$ to be $g_0e_1+e_2$, we may
w.l.o.g. assume $g_0=0$. Thus
\be\label{sumset-for2}\Sigma_{mn}(T)=B:=(mn-n_1)de_1+\{0,de_1,\ldots,(mn
- \ell n-1)de_1\},\ee which is an arithmetic progression of
difference $de_1$ and length $mn-ln$ (in view of $0\notin \Sigma_{nm}(T)$). In view of (\ref{subset-for-S2}), we have $B=A$
with
$$2\leq n \leq |A|=mn - \ell n\leq mn-n\leq mn-2.$$ Thus $de_1=\pm e_1$ (as the difference of an arithmetic progression under the above assumptions is unique up to sign).
Consequently, (\ref{subset-for-S2}) and (\ref{sumset-for2}) imply
that $n_1=nm-1$ if $de_1=e_1$ (since $|S'|\leq 2nm-2$), and that
$n_1=mn - \ell n$ if $de_1=-e_1$ (since $|S'|<2mn - \ell n$).
However, in the former case, $e_2$ has the desired multiplicity in
$S$, while in the latter case, $n_2 = 2mn - \ell n - 1-n_1=mn-1$,
and thus $de_1+e_2=-e_1+e_2$ has the desired multiplicity,
completing the proof.
\end{proof}

\bigskip
\section{Proof of the corollary} \label{6}
\bigskip

Let \ $G = C_{n_1} \oplus C_{n_2}$, \ with $1 < n_1 \t n_2$, and
suppose that, for every prime divisor $p$ of $n_1$, the group \ $C_p
\oplus C_p$ has Property {\bf B}. The assertion  that $C_{n_1}
\oplus C_{n_1}$ has Property {\bf B} follows from the Theorem and
from the following two statements{\rm \,:}
\begin{enumerate}
\item[(a)] For every $n \in [2, 10]$, the group $C_n \oplus C_n$ has
           Property {\bf B}: for $n \le 6$ this may be found in
           \cite[Proposition 4.2]{Ga-Ge03b}; the cases $n \in \{8,9,10\}$ (and
           more) are settled in \cite{Bh-Ha-SP08c}.

\smallskip
\item[(b)] If $n \ge 6$ and $C_n \oplus C_n$ has Property {\bf B},
           then $C_{2n} \oplus C_{2n}$ has Property {\bf B} (see
           \cite[Theorem 8.1]{Ga-Ge03b}).
\end{enumerate}
Since \ $C_{n_1} \oplus C_{n_1}$ \ has Property {\bf B}, the
characterization of the minimal zero-sum sequences over $G$ of
length $\mathsf D (G)$ now follows from \cite[Theorem 3.3]{Sc08e}.\qedsymbol

\bigskip
\noindent {\bf Acknowledgments.}  This work was supported partially
by NSFC with grant no. 10671101 and by the 973 Project with grant no
9732006CB805904. It was further supported by the Austrian Science
Fund FWF (Project Number M1014-N13).


\bibliography{zerosum,fact,ger,hk,ideal}
\bibliographystyle{amsplain}
\end{document}